\documentclass[reqno]{amsart}
\usepackage{latexsym,amsmath,amssymb,amscd}
\usepackage[all]{xy}
\usepackage{tikz-cd}
\usepackage{enumerate}
\usepackage{hyperref} 
\usepackage{mathrsfs}
\usepackage{amsfonts}
\usepackage{setspace}
\usepackage{microtype}

\makeatletter
\@addtoreset{figure}{section}
\def\thefigure{\thesection.\@arabic\c@figure}\def\fps@figure{h,t}
\@addtoreset{table}{bsection}

\def\thetable{\thesection.\@arabic\c@table}
\def\fps@table{h, t}
\@addtoreset{equation}{section}

\makeatother

\makeatletter
\newcommand\@dotsep{4.5}
\def\@tocline#1#2#3#4#5#6#7{\relax
	\ifnum #1>\c@tocdepth 
	\else
	\par \addpenalty\@secpenalty\addvspace{#2}%
	\begingroup \hyphenpenalty\@M
	\@ifempty{#4}{%
		\@tempdima\csname r@tocindent\number#1\endcsname\relax
	}{%
		\@tempdima#4\relax
	}%
	\parindent\z@ \leftskip#3\relax \advance\leftskip\@tempdima\relax
	\rightskip\@pnumwidth plus1em \parfillskip-\@pnumwidth
	#5\leavevmode\hskip-\@tempdima #6\relax
	\leaders\hbox{$\m@th
		\mkern \@dotsep mu\hbox{.}\mkern \@dotsep mu$}\hfill
	\hbox to\@pnumwidth{\@tocpagenum{#7}}\par
	\nobreak
	\endgroup
	\fi}
\makeatother 

\newtheorem{theorem}{Theorem}
\newtheorem{corollary}[theorem]{Corollary}
\newtheorem{definition}[theorem]{Definition}
\newtheorem{example}[theorem]{Example}

\newtheorem{lemma}[theorem]{Lemma}

\newtheorem{proposition}[theorem]{Proposition}
\newtheorem{remark}[theorem]{Remark}

\numberwithin{theorem}{section}
\numberwithin{equation}{section}

\title[]{The $\mathcal{R}$-Shilov boundary for a local operator space}

\author{
  Maria Joi\c ta$^{1}$ \and Gheorghe-Ionu\c{t} \c{S}imon$^{1,*}$ \\
  {\tiny $^1$Department of Mathematics, Faculty of Applied Sciences, University Politehnica of Bucharest, 313 Spl. Independentei, 060042, Bucharest, Romania} \\
}

\thanks{$^*$Corresponding author. Email: \href{mailto:ionutsimon.gh@gmail.com}{ionutsimon.gh@gmail.com}}
\thanks{ Email: \href{mailto:maria.joita@upb.ro}{maria.joita@upb.ro}, \href{mailto:mjoita@fmi.unibuc.ro}{mjoita@fmi.unibuc.ro}, URL: \url{http://sites.google.com/a/g.unibuc.ro/maria-joita}}

\begin{document}

\onehalfspacing
\maketitle

\begin{abstract}
To extend the notion of the injective envelope of a unital operator space to the locally convex case, Dosi in \cite{DD} first introduced the notion of the injective $\mathcal{R}$-envelope for a unital operator space and then defined the injective $\mathcal{R}$-envelope for a unital local operator space as the closure of the injective $\mathcal{R}$-envelope for its bounded part. In this paper, we investigate the existence of the Shilov boundary ideal in
this context, as defined by Arveson in \cite{Arv69}. To do this, by following the conceptual frameworks underlying Hamana's constructions of the injective envelope and the $C^*$-envelope, respectively, we define the notions of the injective $\mathcal{R}$-envelope and the $\mathcal{R}-C^*$-envelope for a unital local operator space.
 Furthermore, we show that the injective $
\mathcal{R}$-envelope construction given by us coincides with the one given by Dosi in \cite{DD}.

\end{abstract}

\vspace{1ex}
\begin{small}

\noindent\textbf{Keywords:} local operator space, locally \( W^* \)-algebras, quantized domain, injective envelope, rigidity.

\noindent\textbf{MSC (2020):} 46L05; 46L07; 46L10; 47L25
\end{small}
\vspace{1.5ex}

\section{Introduction}

A concrete operator space is a subspace of a $C^{\ast }$-algebra. The well-known Gelfand-Naimark theorem states that any unital commutative $C^*$-algebra is $*$-isomorphic to a $C^*$-algebra of continuous functions on a compact Hausdorff space.
Thus, in non-commutative geometry, $C^{\ast }$-algebras are seen as non-commutative analogues of
topological spaces, and operator spaces are viewed as non-commutative analogues of Banach spaces. Many classical
geometric concepts have been reformulated in terms of commutative $C^{\ast }$%
-algebras to obtain generalizations in the non-commutative case.
Arveson \cite{Arv69} introduced the notions of boundary representation for an
operator system (i.e., a unital self-adjoint subspace of a unital $C^{\ast }$
-algebra) as a non-commutative analogue of a Choquet point in the commutative case, and the
Shilov boundary ideal for an operator system as the analogue of the classical Shilov
boundary. The existence of boundary representations remained an open problem for
nearly 40 years. In 2007, Arveson \cite{Arv08} showed that a separable
operator system $\mathcal{S}$ has sufficiently many boundary
representations, and that the Shilov ideal for $\mathcal{S}$ is the intersection
of the kernels of the boundary representations. The full resolution of the existence of the Choquet boundary (also in the non-separable case) is achieved by Davidson-Kennedy \cite{DaKe15}.
 If $\mathcal{S}\subseteq B(%
\mathcal{H})$ is an operator system and $i:\mathcal{S\rightarrow }B(\mathcal{%
K})$ is a complete isometry, then $i\left( \mathcal{S}\right) \subseteq B(%
\mathcal{K})$ is also an operator system with the same matrix order structure as $
\mathcal{S}.$ However, there is generally no canonical relationship between the $C^{\ast }$%
-algebras $C^{\ast }(\mathcal{S})$ and $C^{\ast }(i(\mathcal{S)})$ generated
by $\mathcal{S}$ and $i(\mathcal{S)}$, respectively. 
Hamana \cite{H} showed
that there exists a "minimal " $C^{\ast }$-algebra  generated by an
operator system that is unitally completely order isomorphic to $\mathcal{S}%
.$ This "minimal " $C^{\ast }$-algebra is called the $C^{\ast }$-envelope of 
$\mathcal{S}$, and it is $\ast $-isomorphic to the quotient $C^{\ast }$%
-algebra $C^{\ast }(\mathcal{S})/\mathcal{J},$ where $\mathcal{J}$ is the
Shilov boundary ideal for $\mathcal{S}$ in the sense of Arveson \cite[%
Definition 2.1.3]{Arv69}.

Locally $C^{\ast }$-algebras are generalizations of the notion of $C^{\ast }$%
-algebras. In a locally $C^*$-algebra, the topology is defined by a
family of $C^{\ast }$-seminorms rather than a single $C^{\ast }$-norm. If $%
\mathcal{A}$ is a locally $C^{\ast }$-algebra whose topology is defined by
the family of $C^{\ast }$-seminorms $\{p_{\iota }\}_{\iota \in \Omega }$,
then its bounded part $b(\mathcal{A})=\{a\in \mathcal{A}: \left\Vert
a\right\Vert _{\infty }=\sup \{p_{\iota }\left( a\right) \mid \iota \in \Omega
\}<\infty \}$ is a $C^{\ast }$-algebra that is dense in $\mathcal{A}.$  

A quantized domain is a triple $\{\mathcal{H};\mathcal{E}=\{\mathcal{H}%
_{\iota }\}_{\iota \in \Omega };\mathcal{D}_{\mathcal{E}}=\bigcup\limits_{\iota\in\Omega} \mathcal{H}%
_{\iota }\}$, where $\mathcal{H}$ is a Hilbert space, and $\{\mathcal{H}_{\iota
}\}_{\iota \in \Omega }$ is a family of closed nonzero subspaces of $%
\mathcal{H}$ such that the algebraic sum $\mathcal{D}_{\mathcal{E}}=\bigcup\limits_{\iota\in\Omega} \mathcal{%
H}_{\iota }$ is dense in $\mathcal{H}$. The family  $\{\mathcal{H}_{\iota
}\}_{\iota \in \Omega }$ 
induces a corresponding family of projections $\left\{ P_{\iota }\right\} _{\iota \in \Omega }$ on $\mathcal{H},$ where for each $\iota \in \Omega ,$ $P_{\iota }
$ denotes the orthogonal{\small \ }projection onto $\mathcal{H}_{\iota }$. If $%
\left\{ P_{\iota }\right\} _{\iota \in \Omega }$ is a mutual commuting
family of projections in $B(\mathcal{H})$, we say that $\mathcal{D}_{%
\mathcal{E}}$ is a commutative domain and if $\left\{ P_{\iota }\right\}
_{\iota \in \Omega }$ is an orthogonal family of projections in $B(\mathcal{H%
})$, we say that $\mathcal{D}_{\mathcal{E}}$ is a graded domain. The set
$$
C^{\ast }(\mathcal{D}_{\mathcal{E}}):=\{T\in \mathcal{L}(\mathcal{D}_{
\mathcal{E}}):\ P_{\iota }T\subseteq TP_{\iota },TP_{\iota }=P_{\iota
}TP_{\iota }\in B(\mathcal{H}),\left( \forall \right) \iota \in \Omega \}
$$
is a locally $C^{\ast }$-algebra with the involution given by $
T\mapsto T^{\ast }=\left. T^{\bigstar }\right\vert _{\mathcal{D}_{\mathcal{E}%
}}$, where $T^{\bigstar }$ denotes the adjoint of $T$, and with the topology defined by
the family of $C^{\ast }$-seminorms $\{\left\Vert \cdot \right\Vert _{\iota
}\}_{\iota \in \Omega }$, where $\left\Vert T\right\Vert _{\iota
}=\left\Vert \left. T\right\vert _{\mathcal{H}_{\iota }}\right\Vert =\sup
\{\left\Vert T\left( \xi \right) \right\Vert :\ \xi \in \mathcal{H}_{\iota
},\left\Vert \xi \right\Vert \leq 1\}.$

A locally $C^{\ast }$-algebra $\mathcal{A}$ can be identified with a locally 
$C^{\ast }$-subalgebra of $C^{\ast }(\mathcal{D}_{\mathcal{E}})$ for some
commutative domain $\mathcal{D}_{\mathcal{E}}.$ The locally $C^{\ast }$%
-algebra $C^{\ast }(\mathcal{D}_{\mathcal{E}})$ plays, in a certain sense, the
role of $B(\mathcal{H)}$ in the theory of $C^{\ast }$-algebras. The bounded
part of $C^{\ast }(\mathcal{D}_{%
\mathcal{E}})$, denoted by $b(C^{\ast }(\mathcal{D}_{\mathcal{E}}))$, can be identified with $\{\{P_{\iota }\}_{\iota \in \Omega
}\}^{\prime },$ the commutant in $B(\mathcal{H})$ of the family of projections $\{P_{\iota }\}_{\iota \in \Omega }$.

Effros and Webster \cite{EW} initiated the study of a locally convex version
of operator spaces, known as local operator spaces. In 2008, A. Dosiev \cite%
{D} represented local operator spaces as subspaces of the locally $C^{\ast }$%
-algebra $C^{\ast }(\mathcal{D}_{\mathcal{E}})$ of unbounded operators on a
quantized domain $\mathcal{D}_{\mathcal{E}}$. This result generalizes Ruan's
representation theorem for operator spaces (see, e.g., \cite[Theorem 2.3.5]
{ER}).

A unital local operator space $\mathcal{V}\subseteq  C^{\ast }(\mathcal{D}_{%
\mathcal{E}})$ is injective if, for any unital local operator space $%
\mathcal{W}$ and any unital local operator subspace $\mathcal{W}_{0}\subseteq 
\mathcal{W}$, every unital local completely contractive map from $\mathcal{W}%
_{0}$ to $\mathcal{V}$ extends to a unital local completely contractive map
from $\mathcal{W}$ to $\mathcal{V}$. Dosiev \cite{DD} explores the notion of injectivity
in the locally convex setting, establishing a strong connection with
injectivity in the normed case. If $\mathcal{A}\subseteq C_{\mathcal{E}%
}^{\ast }(\mathcal{D})$ is an injective locally $C^{\ast }$-algebra, then
its bounded part $b(\mathcal{A})$ is an injective $C^{\ast }$-algebra (see 
\cite[Proposition 3.1]{DD}). In general, the converse does not hold: a completely contractive map cannot, in general, be extended to a local completely contractive map.

A locally $W^{\ast }$-algebra (that is, an inverse limit of an inverse
system of $W^{\ast }$-algebras with $w^{\ast }$
-continuous $\ast $-morphisms as connecting maps) can be indentified with a locally von Neumann
algebra $\mathcal{A}$ $\subseteq C^{\ast }(\mathcal{D}_{\mathcal{E}})$ for
some commutative domain $\{\mathcal{H};\mathcal{E=}\{\mathcal{H}_{\iota
}\}_{\iota \in \Omega };\mathcal{D}_{\mathcal{E}}=\bigcup\limits_{\iota\in\Omega} \mathcal{H}_{\iota
}\}$ (see \cite{D1}). Thus, $b(\mathcal{A})$ is a von Neuman algebra and $%
\{P_{\iota }\}_{\iota \in \Omega }\subseteq b(\mathcal{A})\subseteq
\{\{P_{\iota }\}_{\iota \in \Omega }\}^{\prime }.$ Let $\mathcal{R}$ be the
commutative subring of $B(\mathcal{H})$ generated by $\{P_{\iota }\}_{\iota
\in \Omega }\cup \{1_{\mathcal{H}}\}$. Then both $b(\mathcal{A})$ and $\mathcal{%
A}$ carry a natural structure of $\mathcal{R}$-modules. A completely contractive map
on $b(\mathcal{A})$ that is also an $\mathcal{R}$ -morphism extends to a
local completely contractive map on $\mathcal{A}.$ A unital local operator
space $\mathcal{V}\subseteq C^{\ast }(\mathcal{D}_{\mathcal{E}})$ is an
injective local $\mathcal{R}$-module if it is injective and  $\mathcal{%
RV\subseteq V}$ (see \cite{DD}). Dosiev \cite[Proposition 4.1]{DD} shows
that if $C^{\ast }(\mathcal{D}_{\mathcal{E}})$ is injective, then a unital
local operator space $\mathcal{V}\subseteq C^{\ast }(\mathcal{D}_{\mathcal{E}%
})$ is an injective local $\mathcal{R}$-module if and only if $b(\mathcal{V}%
)\subseteq b\left( C^{\ast }(\mathcal{D}_{\mathcal{E}})\right) $ is an
injective\textit{\ }$\mathcal{R}$-module and $\overline{b(\mathcal{V})}=%
\mathcal{V}$ (the closure of $b(\mathcal{V})$ in $C^{\ast }(\mathcal{D}_{%
\mathcal{E}})$). Moreover, if $C^{\ast }(\mathcal{D}_{\mathcal{E}})$ is injective,
then a locally $W^{\ast }$-algebra $\mathcal{A}\subseteq C^{\ast }(\mathcal{D%
}_{\mathcal{E}})$ is injective if and only if $b(\mathcal{V})$ is
injective \cite[Theorem 4.1]{DD}.

The structure of this paper is as follows: In Section 2, we collect some
basic facts about locally $C^{\ast }$-algebras. We refer to \cite{A}, \cite%
{D}, \cite{D1}, \cite{EW}, \cite{F1}, \cite{Fr}, \cite{I}, \cite{P} and \cite%
{S} for more details. Section 3  is divided into two subsections. In Subsection 3.1., we recall the construction of Dosiev 
\cite[Section 3]{DD} of  the injective $\mathcal{R}$-envelope for an operator
space contained in an injective von Neumann algebra. Subsection 3.2 is devoted to the construction of the $\mathcal{R}$-$C^{\ast }$-envelope of a unital operator
space contained in an injective von Neumann algebra. Section 4 consists of two subsections. In Subsection 4.1, we present the
construction of the injective $\mathcal{R}$-envelope of a unital local
operator space, in the sense of Hamana (see Definition \ref{Def-INJ-Env}). 
In addition, in Lemma \ref{Lemma 2}, we establish a locally convex version of
the Hamana-Ruan extension lemma (see \cite[6.1.5]{ER}, \cite{H}, \cite{Ruan}%
). We also show that every unital local operator space has an
injective $\mathcal{R}$-envelope, which is unique up to a unital local
complete order $\mathcal{R}$-isomorphism (Proposition \ref{Coro}). Furthermore, we
show that our definition of the injective $\mathcal{R}$-envelope of a unital
local operator space coincides with the definition given by Dosi in \cite{DD}
(Proposition \ref{Prop. def-equiv}). In Subsection 4.2. we introduce the
concept of the local $\mathcal{R}$-$C^{\ast }$-envelope of a unital local
operator space (Definition \ref{UP-local}) and prove that any unital
local operator space has a (unique) local $\mathcal{R}$-$C^{\ast }$-envelope
(Theorem \ref{Th-C^*-env}, Corollary \ref{Corollary C^*-env}). We also establish a connection between the local $\mathcal{R}$-$C^{\ast }$%
-envelope of a unital local operator space and the $\mathcal{R}$-$C^{\ast }$%
-envelope of a unital operator space, inspired by the one introduced by
Dosi \cite{DD} for the injective $\mathcal{R}$-envelope (see, Remark \ref%
{Remarkk C^*-env}). Finally, in Section 5, we present our main results concerning the existence of the local $\mathcal{R}$-Shilov
boundary ideal of a unital local operator space (Definition %
\ref{Deff Shilov R-ideal}, Theorem \ref{R-Shilov boundary}, and Corollary \ref%
{C. R-Shilov}).

\section{Preliminaries}

\subsection{Locally $C^*$-algebras}

Let $\mathcal{A}$ be a unital $\ast $-algebra with unit denoted by $1_{%
\mathcal{A}}$. A seminorm $p$ on $\mathcal{A}$ is called \textit{%
sub-multiplicative} if $p(1_{\mathcal{A}})=1$ and $p(ab)\leq p(a)p(b)$ for
all $a,b\in \mathcal{A}$. A sub-multiplicative seminorm $p$ on $\mathcal{A}$
is called a \textit{$C^{\ast }$-seminorm} if $p(a^{\ast })=p(a)$ and $%
p(a^{\ast }a)=p(a)^{2}$ for all $a\in \mathcal{A}$.

Let $\left( \Lambda ,\leq \right) $ be a directed poset, and let $\mathcal{P}%
:=\left\{ p_{\lambda }:\lambda \in \Lambda \right\} $ be a family of $%
C^{\ast }$-seminorms defined on the $\ast $-algebra $\mathcal{A}$. We say
that $\mathcal{P}$ is an \textit{upward filtered family} of $C^{\ast }$%
-seminorms if $p_{\lambda _{1}}(a)\leq p_{\lambda _{2}}(a)$ for all $a\in 
\mathcal{A}$, whenever $\lambda _{1}\leq \lambda _{2}\ $in$\ \Lambda .$

\begin{definition}
A locally $C^*$-algebra $\mathcal{A}$ is a complete Hausdorff topological $*$-algebra over $\mathbb{C}$
whose topology is determined by an upward filtered family $\left\lbrace p_{\lambda}: \lambda\in\Lambda  \right\rbrace$ of $C^*$-seminorms defined on $\mathcal{A}$.
\end{definition}

An element $a\in A$ is \textit{bounded} if $\sup\limits\left\lbrace p_{\lambda}(a)\mid \lambda\in\Lambda \right\rbrace<\infty.$ The subset $b(A)=\left\lbrace  a\in A\mid
\Vert a \Vert_{\infty}:=\sup\limits\left\lbrace p_{\lambda}(a)\mid \lambda\in\Lambda \right\rbrace<\infty \right\rbrace$ is a $C^*$-algebra with respect to the $C^*$-norm $\Vert\cdot\Vert_{\infty}.$ Moreover, $b(A)$ is dense in $\mathcal{A}$.

We see that $\mathcal{A}$ can be realized as a projective limit of an inverse family of $C^*$-algebras as follows:  

For each $\lambda \in \Lambda $, let $\mathcal{I}_{\lambda }:=\left\{ a\in
A:p_{\lambda }(a)=0\right\} .$ Clearly, $\mathcal{I}_{\lambda }$ is a closed
two-sided $\ast $-ideal in $\mathcal{A}$ and $\mathcal{A}_{\lambda }:=%
\mathcal{A}/\mathcal{I}_{\lambda }$ is a $C^{\ast }$-algebra with respect to
the norm induced by $p_{\lambda }$ (see \cite{A}). The canonical quotient $%
\ast $-homomorphism from $\mathcal{A}$ to $\mathcal{A}_{\lambda }$ is
denoted by $\pi _{\lambda }^{\mathcal{A}}.$ For each $\lambda _{1},\lambda
_{2}\in \Lambda $ with $\lambda _{1}\leq \lambda _{2}$, there is a canonical
surjective $\ast $-homomorphism $\pi _{\lambda _{2}\lambda _{1}}^{\mathcal{A}%
}:\mathcal{A}_{\lambda _{2}}\rightarrow \mathcal{A}_{\lambda _{1}}$ defined
by $\pi _{\lambda _{2}\lambda _{1}}^{\mathcal{A}}\left( a+\mathcal{I}%
_{\lambda _{2}}\right) =a+\mathcal{I}_{\lambda _{1}}$, for all $a\in 
\mathcal{A}.$ Then $\left\{ \mathcal{A}_{\lambda },\pi _{\lambda _{2}\lambda
_{1}}^{\mathcal{A}}\right\} $ forms an inverse system of $C^{\ast }$%
-algebras, because $\pi _{\lambda _{1}}^{\mathcal{A}}=\pi _{\lambda
_{2}\lambda _{1}}^{\mathcal{A}}\circ \pi _{\lambda _{2}}^{\mathcal{A}}$
whenever $\lambda _{1}\leq \lambda _{2}.$ The projective limit 
\begin{equation*}
\varprojlim\limits_{\lambda }\mathcal{A}_{\lambda }:=\left\{ \{a_{\lambda
}\}_{\lambda \in \Lambda }\in \prod\limits_{\lambda \in \Lambda }\mid \pi
_{\lambda _{2}\lambda _{1}}^{\mathcal{A}}(a_{\lambda _{2}})=a_{\lambda
_{1}}\ \mathit{whenever}\ \lambda _{1}\leq \lambda _{2},\lambda _{1},\lambda
_{2}\in \Lambda \right\}
\end{equation*}%
of the inverse system of $C^{\ast }$-algebras $\left\{ \mathcal{A}_{\lambda
},\pi _{\lambda _{2}\lambda _{1}}^{\mathcal{A}}\right\} $ is a locally $%
C^{\ast }$-algebra that may be identified with $\mathcal{A}$ by the map $%
a\mapsto \left( \pi _{\lambda }^{\mathcal{A}}(a)\right) _{\lambda \in
\Lambda }.$
\medskip

Subsequently, we recall local positivity in $\mathcal{A}$ with respect to the family $\left\lbrace p_{\lambda}: \lambda\in\Lambda  \right\rbrace$ of $C^*$-seminorms, as in \cite{D}.

\begin{definition}
Let $\mathcal{A}$ be a locally $C^*$-algebra. An element $a\in \mathcal{A}$ is called 
\begin{itemize}
\item[i)] local self-adjoint (or local hermitian) if $a=a^*+x$, where $x\in \mathcal{A}$ such that $p_{\lambda}(x)=0$ for some $\lambda\in\Lambda$.
\item[ii)] local positive if $a=b^*b+x$, where $b,x\in \mathcal{A}$ such that $p_{\lambda}(x)=0$ for some $\lambda\in\Lambda.$
\end{itemize}
\end{definition}

In this case, we say that $a$ is $\lambda$-\textit{self-adjoint} (or \textit{$\lambda$-hermitian}) and \textit{$\lambda$-positive}, respectively. We write $a\geq_{\lambda}0$ when $a$ is $\lambda$-positive and $a=_{\lambda}0$ whenever $p_{\lambda}(a)=0.$ Note that $a\in \mathcal{A}$ is local self-adjoint if and only if there exists $\lambda\in\Lambda$ such that $\pi_{\lambda}^{\mathcal{A}}(a)$ is self-adjoint in the $C^*$-algebra $\mathcal{A}_{\lambda}$. Similarly, $a\in \mathcal{A}$ is local positive if and only if there exists $\lambda\in\Lambda$ such that $\pi_{\lambda}^{\mathcal{A}}(a)$ is positive in $\mathcal{A}_{\lambda}.$

\begin{definition}
Let $\mathcal{A}$ and $\mathcal{B}$ be two locally $C^{\ast }$-algebras,
where the topology is given by the families of $C^{\ast }$-seminorms $%
\{p_{\lambda }\}_{\lambda \in \Lambda }$ and $\{q_{\delta }\}_{\delta \in
\Delta }$, respectively. A $\ast $-morphism $\pi :\mathcal{A\rightarrow B}$
is called a \textit{local contractive $\ast $-morphism} if, for every $%
\delta \in \Delta $, there exists $\lambda \in \Lambda $ such that $%
q_{\delta }(\pi (a))\leq p_{\lambda }(a)$, for all $a\in \mathcal{A}.$ If $%
\Delta =\Lambda $ and $q_{\lambda }(\pi (a))=p_{\lambda }(a)$ for all $a\in 
\mathcal{A}$, then $\pi $ is called a \textit{local isometric $\ast $%
-morphism}.

\end{definition}

\subsection{Quantized domains}

A quantized domain is a triple $\{\mathcal{H};\mathcal{E=}\{\mathcal{H}
_{\iota }\}_{\iota \in \Omega };\mathcal{D}_{\mathcal{E}}=\bigcup\limits_{\iota\in\Omega} \mathcal{H}%
_{\iota }\}$, where $\mathcal{H}$ is a Hilbert space, $\{\mathcal{H}_{\iota
}\}_{\iota \in \Omega }$ is a family of closed nonzero subspaces of $%
\mathcal{H}$ and the algebraic sum $\mathcal{D}_{\mathcal{E}}$ is dense in $%
\mathcal{H}$. The family $\{\mathcal{H}_{\iota }\}_{\iota \in \Omega }$
of closed nonzero subspaces of $\mathcal{H}$ induces a corresponding family  $\left\{
P_{\iota }\right\} _{\iota \in \Omega }$ of projections on $\mathcal{H},$
where for each $\iota \in \Omega ,$ $P_{\iota }$ denotes the orthogonal
projection on $\mathcal{H}_{\iota }$. If $\left\{ P_{\iota }\right\} _{\iota
\in \Omega }$ is a mutually commuting family of projections in $B(\mathcal{H})$%
, we say that $\mathcal{D}_{\mathcal{E}}$ is a \textit{commutative domain}.
If $\left\{ P_{\iota }\right\} _{\iota \in \Omega }$ is an orthogonal
family of projections in $B(\mathcal{H})$, we say that $\mathcal{D}_{%
\mathcal{E}}$ is a \textit{graded domain. }

We may assume that $\left( \Omega ,\leq \right) $ is a directed poset and that $
\mathcal{E=}\{\mathcal{H}_{\iota }\}_{\iota \in \Omega }$ is an upward
filtered family of closed subspaces of a Hilbert space $\mathcal{H}$ such that the union space $\mathcal{D}_{%
\mathcal{E}}:=\bigcup\limits_{\iota \in \Omega }\mathcal{H}_{\iota }$ is
dense in $\mathcal{H}$. Indeed, let $\Omega ^{f}=\{\alpha \subseteq \Omega
\mid \alpha $ is finite$\}.$ For each $\alpha \in $ $\Omega ^{f},$ define $\mathcal{H}%
_{\alpha }:=\bigcup\limits_{\iota \in \alpha }\mathcal{H}_{\iota }.$ Then $\{%
\mathcal{H}_{\alpha }\}_{\alpha \in \Omega ^{f}}$ is an upward filtered
family of closed subspaces, and $\mathcal{D}_{\mathcal{E}}:=\bigcup\limits_{%
\alpha \in \Omega ^{f}}\mathcal{H}_{\alpha }.$ If $\Omega $ is countable, we
say that $\lbrace\mathcal{H};\mathcal{E}=\{\mathcal{H}_{\iota }\}_{\iota \in \Omega
};\mathcal{D}_{\mathcal{E}}=\bigcup\limits_{\iota\in\Omega} \mathcal{H}_{\iota }\rbrace$ is a \textit{Fr\'{e}chet quantized domain}.

Let $\left\{ \mathcal{H},\mathcal{E},\mathcal{D}_{\mathcal{E}}\right\} $ be
a quantized domain. The set 
\begin{equation*}
C^{\ast }(\mathcal{D}_{\mathcal{E}}):=\{T\in \mathcal{L}(\mathcal{D}_{%
\mathcal{E}}):\ P_{\iota }T\subseteq TP_{\iota },TP_{\iota }=P_{\iota
}TP_{\iota }\in B(\mathcal{H}), \left( \forall \right) \iota \in \Omega \},
\end{equation*}%
where $\mathcal{L}(\mathcal{D}_{\mathcal{E}})$ denotes the set of all linear
operators on $\mathcal{D}_{\mathcal{E}}$, is a locally $C^{\ast }$-algebra
with  involution defined by
\begin{equation*}
T^{\ast }:=T^{\bigstar }\restriction_{\mathcal{D}_{\mathcal{E}}}\in C^{\ast }(\mathcal{%
D}_{\mathcal{E}})
\end{equation*}%
where $T^{\bigstar }$ is the adjoint operator associated to $T,$ and 
topology induced by the family of $C^{\ast }$-seminorms $\{q_{\iota
}\}_{\iota \in \Omega }$, 
\begin{equation*}
q_{\iota }(T):=\Vert T\upharpoonright _{\mathcal{H}_{\iota }}\Vert _{B(%
\mathcal{H}_{\iota })}.
\end{equation*}

Note that
\begin{itemize}
\item[i)] $b\left( C^*_{\mathcal{E}}(\mathcal{D}) \right)$ is identified with the $C^*$-algebra $$\left\lbrace T\in B(\mathcal{H})\mid P_{\iota}T=TP_{\iota}, (\forall)\ \iota\in\Omega  \right\rbrace$$ via the map $T\mapsto\tilde{T}$, where $\tilde{T}$ is the extension of $T$ to $\mathcal{H}$ (see \cite[Lemma 3.1.]{D1})).
\item[ii)] For each $\iota\in\Omega, \left( C^*_{\mathcal{E}}(\mathcal{D})\right)_{\iota} $ is a $C^*$-subalgebra of $B(\mathcal{H}_{\iota})$ and $$\pi_{\iota}^{C^*_{\mathcal{E}}(\mathcal{D})}\left(  T \right)= T\restriction_{\mathcal{H}_{\iota}}.$$ Moreover, $T\geq_{\iota} 0$ if and only if $T\restriction_{\mathcal{H}_{\iota}}\geq 0$ in $B(\mathcal{H}_{\iota}).$ 
\end{itemize}


 As mentioned in \cite{D}, the locally $C^*$-algebra $C^{\ast }(\mathcal{D}_{\mathcal{E}})$ is considered the natural analog of $B(\mathcal{H})$ in the local operator space theory. However, it should be noted that this algebra works differently than $B(\mathcal{H})$ in the operator space theory, and more careful consideration should be given.

\begin{theorem}[\cite{D1}, Proposition 3.1]
 For any unital locally $C^{\ast }$-algebra $
\mathcal{A}$, there exist a commutative domain $\{\mathcal{H};\mathcal{E};
\mathcal{D}_{\mathcal{E}}\}$ and a local isometric $\ast $-morphism $\pi :
\mathcal{A}\rightarrow C^{\ast }(\mathcal{D}_{\mathcal{E}})$.

\end{theorem}

 This result can be regarded as an unbounded analogue of the Gelfand-Naimark theorem.

\begin{remark}
Let $\{\mathcal{H};\mathcal{E=}\{\mathcal{H}_{\iota }\}_{\iota \in \Omega };%
\mathcal{D}_{\mathcal{E}}=\bigcup\limits_{\iota\in\Omega} \mathcal{H}_{\iota }\}$ be a quantized
domain. Then the triple $\{\mathcal{K};\mathcal{F=}\{\mathcal{H}_{\iota }\}_{\iota \in
\Omega };\mathcal{D}_{\mathcal{F}}=\bigcup\limits_{\iota\in\Omega} \mathcal{H}_{\iota }\},$ where $%
\mathcal{K}=\bigoplus\limits_{\iota\in\Omega} \mathcal{H}_{\iota },$
is a graded\textit{\ } domain, and the map $
i:C^{\ast }(\mathcal{D}_{\mathcal{E}})\rightarrow C^{\ast }(\mathcal{D}_{%
\mathcal{F}})\ $ defined by $i\left( T \right) :=\left( T\restriction_{\mathcal{H}_{\iota}} \right) _{\iota \in \Omega }$ is a local isometric $%
\ast $-morphism (see the proof of  \cite[Proposition 3.2]{DD}).

\end{remark}

\subsection{Local positive maps}

Let $\mathcal{A}$ be a unital locally $C^*$-algebra with the topology defined by the family of $C^*$-seminorms $\lbrace p_{\lambda} \rbrace_{\lambda\in\Lambda}$, and let $\mathcal{B}$ be another unital locally $C^*$-algebra with the topology defined by the family of $C^*$-seminorms $\lbrace q_{\delta} \rbrace_{\delta\in\Delta}$.

 For each $n\in\mathbb{N}$, $M_{n}(\mathcal{A})$ denotes the set of all $n\times n$ matrices over $\mathcal{A}$. Note that $M_{n}(\mathcal{A})$ is a unital locally $C^*$-algebra with the associated family of $C^*$-seminorms denoted by 
$\left\lbrace p_{\lambda}^{n}: \ \lambda\in\Lambda  \right\rbrace$, where $p_{\lambda}^{n}\left( [a_{ij}]_{i,j=1}^{n} \right)=\Vert [\pi_{\lambda}^{\mathcal{A}}(a_{ij})]_{i,j=1}^{n} \Vert_{M_{n}(\mathcal{A}_{n})}$ for all $\lambda\in\Lambda$. Moreover, $M_{n}(\mathcal{A})$ can be identified with $\lim\limits_{\longleftarrow}M_{n}(\mathcal{A}_{\lambda}).$

For each $n\in\mathbb{N}$, the \textit{$n$-amplification} of the linear map $\varphi:\mathcal{A}\rightarrow \mathcal{B}$ is the map $\varphi^{(n)}:M_{n}(\mathcal{A})\rightarrow M_{n}(\mathcal{B})$ defined by $$\varphi^{(n)}\left( [a_{ij}]_{i,j=1}^{n} \right):=\left[ \varphi(a_{ij})  \right]_{i,j=1}^{n}$$  for  all $ \left[ a_{ij} \right]_{i,j=1}^{n}\in M_{n}(\mathcal{A})$.

 A linear map $\varphi:\mathcal{A}\rightarrow \mathcal{B}$ is called 
 \textit{local completely positive} if for each $\delta\in\Delta$, there exists $\lambda\in\Lambda$ such that $[\varphi(a_{ij})]\geq_{\delta} 0$ whenever $[a_{ij}]\geq_{\lambda} 0$ and $[\varphi(a_{ij})]=_{\delta} 0$ if $[a_{ij}]=_{\lambda} 0,$ for all $n\in\mathbb{N}$.

\begin{lemma}\label{Lem}
Let $\mathcal{A}$ and $\mathcal{B}$ be two unital locally $%
C^{\ast }$-algebras, $\varphi :\mathcal{A\rightarrow B}$ and $\mathcal{\psi
:B\rightarrow A}$ be unital local completely positive maps such that $%
\varphi \circ \psi =$id$_{\mathcal{B}}.$ Then $\varphi \left( \psi (b)^{\ast
}\psi (b)\right) =b^{\ast }b$ and $\varphi \left( \psi (b)\psi (b)^{\ast
}\right) =bb^{\ast }$, for all $b\in \mathcal{B}$.
\end{lemma}

\begin{proof}
The result follows from \cite[Theorem 7.2]{D} and \cite[Corollary 5.5]{D}.
\end{proof}

The result below may be regarded as a locally convex analogue of \cite[Theorem 2.1]{Gheondea}.

\begin{theorem}[\cite{D}, Corollary 5.5]\label{Mult.Dom}
Let $\varphi :\mathcal{A\rightarrow B}$ be a unital local
completely positive map. Then the following statements hold:

\begin{itemize}
\item[(1)] (The Schwarz Inequality) $\varphi (a)^{\ast }\varphi (a)\leq
\varphi (aa^{\ast })$ for all $a\in \mathcal{A}.$

\item[(2)] (The Multiplicativity Property) Let $a\in \mathcal{A}$. Then

\begin{itemize}
\item[(i)] $\varphi (a)^{\ast }\varphi (a)=\varphi (a^{\ast }a)$ if and only
if $\varphi (ba)=\varphi (b)\varphi (a)$ for all $b\in \mathcal{A}.$

\item[(ii)] $\varphi (a)\varphi (a)^{\ast }=\varphi (aa^{\ast })$ if and
only if $\varphi (ab)=\varphi (a)\varphi (b)$ for all $b\in \mathcal{A}$.
\end{itemize}

\item[(3)] The set $\mathcal{M}_{\varphi }:=\left\{ a\in A\mid \varphi
(a)^{\ast }\varphi (a)=\varphi (a^{\ast }a)\ \mathit{and}\ \varphi
(a)\varphi (a)^{\ast }=\varphi (aa^{\ast })\right\} $ is a unital locally $%
C^{\ast }$-subalgebra of $\mathcal{A}$, and it coincides with the largest
locally $C^{\ast }$-subalgebra $\mathcal{C}$ of $\mathcal{A}$ such that the
restricted map $\varphi \upharpoonright _{\mathcal{C}}:\mathcal{C\rightarrow
B}$ is a unital local contractive $\ast $-morphism. Moreover, $\varphi
(bac)=\varphi (b)\varphi (a)\varphi (c)$ for all $b,c\in \mathcal{M}%
_{\varphi }$ and for all $a\in \mathcal{A}.$
\end{itemize}
\end{theorem}

\subsection{Injective local operator spaces}

Let $\mathcal{A}$ be a unital locally $C^{\ast }$-algebra with the topology
defined by the family of $C^{\ast }$-seminorms $\{p_{\lambda }\}_{\lambda
\in \Lambda }$, and let $\mathcal{B}$ be a unital locally $C^{\ast }$-algebra
with the topology defined by the family of $C^{\ast }$-seminorms $%
\{q_{\delta }\}_{\delta \in \Delta }$.

A unital subspace $\mathcal{V}$ of $\mathcal{A}$ is called a \textit{unital
local operator space} (or, unital quantum space, according to \cite{DD}) .
Let $\mathcal{V\subseteq A}$ and $\mathcal{W\subseteq B}$ be two unital
local operator spaces. A unital linear map $\varphi :\mathcal{V\rightarrow W}
$ is called 
 \begin{itemize}
 \item[i)] \textit{local completely contractive}  if for each $\delta\in\Delta$, there exists $\lambda \in\Lambda$ such that $$q^{n}_{\delta}\left([\varphi(a_{ij})] \right)\leq p^{n}_{\lambda}\left([a_{ij}] \right)$$ for all $[a_{ij}]\in M_{n}(\mathcal{V})$ and for all $n\in\mathbb{N}.$
 \item[ii)] \textit{local completely isometric} if $\Delta=\Lambda$ and $$q^{n}_{\delta}\left([\varphi(a_{ij})] \right)= p^{n}_{\lambda}\left([a_{ij}] \right)$$ for all $[a_{ij}]\in M_{n}(\mathcal{V})$ and for all $n\in\mathbb{N}.$
 
  \end{itemize}

 Note that, if $\varphi:\mathcal{V} \rightarrow \mathcal{W}$ is a unital local completely isometric map, then $\varphi(\mathcal{V})$ is a unital local operator space and $\varphi^{-1}:\varphi(\mathcal{W})\rightarrow \mathcal{V}$ is a unital local completely isometric map.

A unital linear map $\varphi:\mathcal{V} \rightarrow \mathcal{W}$ is called a   \textit{unital local complete order isomorphism} if $\varphi$ is bijective and both  $\varphi$ and $\varphi^{-1}$ are local completely isometric maps.
\medskip


A unital local operator space $\mathcal{V\subseteq A}$ is \textit{injective}, if for
any unital local operator space $\mathcal{W}$ and any unital local operator
subspace $\mathcal{W}_{0}\subseteq \mathcal{W}$, any unital local completely
contractive map from $\mathcal{W}_{0}$ to $\mathcal{V}$ extends to a unital
local completely contractive map from $\mathcal{W}$ to $\mathcal{V}$.

If $\{\mathcal{H};\mathcal{E=}\{\mathcal{H}_{\iota }\}_{\iota \in \Omega };%
\mathcal{D}_{\mathcal{E}}=\bigcup\limits_{\iota\in\Omega} \mathcal{H}_{\iota }\}$ is a graded domain,
then $C^{\ast }(\mathcal{D}_{\mathcal{E}})$ is an injective local operator
space \cite[Lemma 3.1]{DD}.

\section{The $\mathcal{R}-C^{\ast }$ -envelope of an operator space}

Let $\mathcal{A}\subseteq B(\mathcal{H})$ be a unital injective von Neumann
algebra on a Hilbert space $\mathcal{H}$. Define $\mathcal{R}$ as the
commutative subring of $\mathcal{A}$ generated by its central projections
and id$_{\mathcal{H}}$. Since $\mathcal{R}\mathcal{A}=\mathcal{A}\mathcal{R}%
\subseteq \mathcal{A}$, the algebra $\mathcal{A}$ is equipped with the
canonical $\mathcal{R}$-module structure, making it an (algebraic) $\mathcal{%
R}$-module. A linear map $\varphi :\mathcal{A}\rightarrow \mathcal{A}$ is called an $%
\mathcal{R}$-map if $\varphi \left( ea\right) =e\varphi \left( a\right) $,
for all $a\in \mathcal{A}$ and for all central projection $e$ of $\mathcal{A}
$. We denote by $B_{1}^{\mathcal{R}}\left( \mathcal{A}\right)$ the set of
all unital completely contractive $\mathcal{R}$\textit{-}maps\textit{\ }$%
\varphi :\mathcal{A}\rightarrow \mathcal{A}$. A unital completely
contractive map $\varphi :\mathcal{A}\rightarrow \mathcal{A}$ which is also
an $\mathcal{R}$-map and satisfies $\varphi \circ \varphi =\varphi $ is called an $%
\mathcal{R}$\textit{-projection}.

A unital subspace $\mathcal{V}\subseteq \mathcal{A}$ is called a unital
operator space. If $\mathcal{VR\subseteq V}$, we say that $\mathcal{V}$ is an $\mathcal{R}
$ -module.

Let $\mathcal{V\subseteq A}$ and $\mathcal{W\subseteq A}$ be two unital $%
\mathcal{R}$-modules. We say that a unital linear map $\varphi :\mathcal{%
V\rightarrow W}$ is a \textit{unital complete order }$\mathcal{R}$\textit{%
-isomorphism} if $\varphi $ is bijective and $\varphi $ and both $\varphi ^{-1}$
are both completely isometric $\mathcal{R}$-maps.

\subsection{The injective $\mathcal{R}$-envelope of an operator space 
\protect\cite{DD}}

A unital subspace $\mathcal{V}\subseteq \mathcal{A}$ is \textit{an injective 
}$\mathcal{R}$\textit{-module} if $\mathcal{V}$ is injective and $\mathcal{%
VR\subseteq V}$. A unital operator space $\mathcal{V}\subseteq \mathcal{A}$
is an injective $\mathcal{R}$\textit{-}module if and only if there exists an 
$\mathcal{R}$-projection $\varphi :\mathcal{A}\rightarrow \mathcal{A}$ such
that $\mathcal{V}=$Im$\left( \varphi \right) $, the image of $\varphi $ \cite
[Lemma 3.2]{DD}$.$

Let $\mathcal{V}\subseteq \mathcal{A}$ be a unital operator space. An $%
\mathcal{R}$\textit{-extension }of $\mathcal{V}$ is a pair $\left( \mathcal{W%
},i\right) $, where $\mathcal{W}\subseteq \mathcal{A}$ is an $\mathcal{R}$%
-module and $i$ is a complete isometry. An $\mathcal{R}$\textit{-extension 
}$\left( \mathcal{W},i\right) $ of $\mathcal{V}$ is called \textit{rigid} if the
identity on $\mathcal{W}$ is the unique completely contractive $\mathcal{R}$%
-map that extends the identity map on $i\left( \mathcal{V}\right) $.

An \textit{injective} $\mathcal{R}$\textit{-envelope} for a unital operator
space $\mathcal{V\subseteq A}$ is an $\mathcal{R}$-rigid extension $\left( 
\mathcal{W},i\right) $ of $\mathcal{V}$ such that $\mathcal{W}$ is an
injective $\mathcal{R}$-module.

If $\varphi $ is an $\mathcal{R}$-projection such that $\left. \varphi
\right\vert _{\mathcal{V}}=$id$_{\mathcal{V}},$ then $\left( \text{Im}\left(
\varphi \right) ,i\right) $ is an injective $\mathcal{R}$-extension of $%
\mathcal{V}$.

Dosiev \cite{DD} introduces a partial order and an equivalence relation on
the set $ \mathscr{E}_{V}:=\left\{ \varphi \in B_{1}^{\mathcal{R}}\left(
A\right) \mid \varphi \upharpoonright _{V}=\text{id}_{V}\right\} $ as
follows: $\varphi _{1}\preceq \varphi _{2}$ if and only if $\Vert \varphi
_{1}(a)\Vert \leq \Vert \varphi _{2}(a)\Vert $ for all $a\in A$, and $\varphi
_{1}\sim \varphi _{2}$ if and only if $\varphi _{1}\preceq \varphi _{2}$ and 
$\varphi _{2}\preceq \varphi _{1}$. An element $\varphi \in \mathscr{E}_{\mathcal{V}}$ is minimal if $\psi \preceq \varphi $ implies $\psi \sim
\varphi .$ The set of all minimal elements in $\mathscr{E}_{\mathcal{V}}$ is
denoted by $\min \mathscr{E}_{\mathcal{V}}.$

\begin{proposition}\label{Prop}
\cite[Lemma 3.3]{DD} Let $\mathcal{V\subseteq A}$ be a unital operator space. Then $\min \mathscr{E}_{\mathcal{V}}$ is non-empty and consists of
projections in $\mathscr{E}_{\mathcal{V}}$. Moreover, for $\varphi\in \min \mathscr{E}_{\mathcal{V}}$ the extension $\left( 
\text{Im}\left( \varphi \right) ,i\right) $ of $\mathcal{V}$ is $\mathcal{R}$%
-rigid, and if $\left( \mathcal{W},i\right) $ is an $\mathcal{R}$-rigid
extension such that $\mathcal{W}$ is an \textit{injective} $\mathcal{R}$%
\textit{-module, then }$\mathcal{W}=$Im$\left( \varphi \right) $ up to a 
\textit{unital complete order }$\mathcal{R}$\textit{-isomorphism.}
\end{proposition}

Therefore, any unital operator space $\mathcal{V\subseteq A}$ has a unique
injective $\mathcal{R}$\textit{-}envelope, denoted by $\mathcal{I}_{\mathcal{R}}(\mathcal{V}).$

\subsection{The $\mathcal{R}-C^{\ast }$-envelope of an operator space}

Let $\mathcal{A}\subseteq B(\mathcal{H})$ be a unital injective von Neumann algebra and  $\mathcal{V\subseteq A}$ be a unital operator subspace. An $\mathcal{R}-C^{\ast }$-\textit{extension} of $\mathcal{V}$ is an $\mathcal{R}$-extension $\left( 
\mathcal{B},\kappa\right) $ of $\mathcal{V}$ such that $\mathcal{B}$ is the $C^{\ast }$-algebra generated by $\mathcal{R}\kappa\left( \mathcal{V}\right) $.

\begin{definition}
An\textit{\ }$\mathcal{R}-C^{\ast }$-\textit{envelope} of $\mathcal{V}$ is an $%
\mathcal{R}-C^{\ast }$-extension $\left( \mathcal{B},\kappa\right) $ with the
following universal property: Given any $\mathcal{R}-C^{\ast }$-extension $%
\left( \mathcal{C},i\right) $ of $\mathcal{V}$, there exists a unique unital
surjective $\mathcal{R}-C^{\ast }$-morphism $\pi :\mathcal{%
C\rightarrow B}$ such that $\pi \circ i=\kappa .$
\end{definition}

\begin{remark}\label{Rem1}
Let $\left( \mathcal{B},i\right) $ be an $\mathcal{R}-C^{\ast }$%
-extension of $\mathcal{V}\subseteq\mathcal{A}$.
Then, $\left( \mathcal{B},\odot \right) $ is a unital $C^{\ast }$%
-algebra. We may regard $\mathcal{B}$ as a unital  operator subspace of 
$\mathcal{A}$, and thus we can associate with it
its injective $\mathcal{R}$-envelope $\left( \mathcal{I}_{\mathcal{R}}(%
\mathcal{B}),\widetilde{\kappa }\right) $. Hence, by Proposition \ref{Prop},
we have $\mathcal{I}_{\mathcal{R}}(\mathcal{B})=$Im$(\varphi )$, where $%
\varphi \in \min \mathscr{E}_{\mathcal{B}}$, and $\left( \mathcal{I}_{\mathcal{R%
}}(\mathcal{B}),\odot _{\varphi }\right) $ becomes a unital $C^{\ast
}$-algebra \cite[Theorem 1.3.13]{Ble}. By $\mathcal{R}$-rigidity, the
identity map id$_{\mathcal{I}_{\mathcal{R}}(\mathcal{B})}$ is the unique
 completely contractive $\mathcal{R}$-map that makes the following
diagram commute.%
\begin{equation*}
\begin{tikzcd} \mathcal{B} && {\mathcal{I}_{\mathcal{R}}(\mathcal{B})} \\ \mathcal{B} && {\mathcal{I}_{\mathcal{R}}(\mathcal{B})}
\arrow["{\widetilde{\kappa}}", from=1-1, to=1-3] \arrow["{id_{\mathcal{B}}}"',
from=1-1, to=2-1] \arrow["{id_{\mathcal{I}_{\mathcal{R}}(\mathcal{B})}}", from=1-3, to=2-3]
\arrow["{\widetilde{\kappa}}", from=2-1, to=2-3] \end{tikzcd}
\end{equation*}

Hence, $\odot =\odot _{\varphi }$ and $\widetilde{\kappa }$ is a unital completely isometric $\ast $-morphism.
\end{remark}

\begin{theorem}
\label{1}  Any unital operator subspace  $\mathcal{V\subseteq A}$ has an $\mathcal{R}-C^{\ast }$-envelope. Moreover, it is unique up to a unital $\mathcal{R}-C^{\ast }$-isomorphism.
\end{theorem}

\begin{proof}
Let $\left( \mathcal{I}_{\mathcal{R}}(\mathcal{V}),\kappa \right) $ be an
injective $\mathcal{R}$-envelope of $\mathcal{V}$. By \cite[Lemma 3.3]{DD}, there
exists $\psi \in \min \mathscr{E}_{\mathcal{V}}$ such that $\mathcal{I}_{%
\mathcal{R}}(\mathcal{V})=$Im$({\psi }).$ Moreover, $\left( \mathcal{I}_{%
\mathcal{R}}(\mathcal{V}),\odot _{\psi }\right) $ is a unital $C^{\ast }$%
-algebra with the multiplication given by $v_{1}\odot _{\psi }v_{2}:=\psi
(v_{1}v_{2})$, the involution and the $C^{\ast }$-norm are inherited from $%
\mathcal{A}$ \cite[Theorem 1.3.13]{Ble}, and it is also an $\mathcal{R}$-module.

Let $C_{e\mathcal{R}}^{\ast }(\mathcal{V}):=C^{\ast }(\mathcal{R}\kappa (%
\mathcal{V}))\subseteq \mathcal{I}_{\mathcal{R}}(\mathcal{V})$ be the $C^{\ast
}$-algebra generated by $\mathcal{R}\kappa (\mathcal{V}).$ Then the pair $\left( C_{e%
\mathcal{R}}^{\ast }(\mathcal{V}),\kappa \right) $ is an $\mathcal{R}%
-C^{\ast }$-extension of $\mathcal{V}.$ We will show that this extension satisfies the
universal property.

Let $\left( \mathcal{B},i\right) $ be an $\mathcal{R}-C^{\ast }$-extension
of $\mathcal{V}$, and let $\left( \mathcal{I}_{\mathcal{R}}(\mathcal{B}),\tilde{\kappa}%
\right) $ be an injective $\mathcal{R}$-envelope of $\mathcal{B}$. By \cite[%
Lemma 3.3]{DD}, there exists $\varphi \in \min \mathscr{E}_{\mathcal{B}}$ such
that $\mathcal{I}_{\mathcal{R}}(\mathcal{B})=$Im$(\varphi )$. Then the pair $\left( 
\mathcal{I}_{\mathcal{R}}(\mathcal{B}),\odot _{\varphi }\right) $ is a
unital $C^{\ast }$-algebra and also an $\mathcal{R}$-module, while $\mathcal{B}$
is a unital $C^{\ast }$-subalgebra of $\mathcal{I}_{\mathcal{R}}(\mathcal{B}%
).$

Since $\mathcal{I}_{\mathcal{R}}(\mathcal{V})$ is an injective $\mathcal{R}$%
-module, the map $\kappa \circ i^{-1}:i(\mathcal{V})\rightarrow \mathcal{I}_{%
\mathcal{R}}(\mathcal{V})$ can be extended to a unital completely
contractive $\mathcal{R}$-map $\pi :\mathcal{B}\rightarrow \mathcal{I}_{%
\mathcal{R}}(\mathcal{V})$ such that $\pi \circ i=\kappa .$ Similarly, the map $\pi
\circ \widetilde{\kappa }^{-1}:\widetilde{\kappa }(\mathcal{B})\subseteq 
\mathcal{I}_{\mathcal{R}}(\mathcal{B})\rightarrow \mathcal{I}_{\mathcal{R}}(%
\mathcal{V})$ can be extended to a unital completely contractive $\mathcal{R}
$-map $\widetilde{\pi }:\mathcal{I}_{\mathcal{R}}(\mathcal{B})\rightarrow 
\mathcal{I}_{\mathcal{R}}(\mathcal{V})$ such that $\widetilde{\pi }\circ 
\widetilde{\kappa }=\pi .$

Since $\mathcal{I}_{\mathcal{R}}(\mathcal{B})$ is an injective  $%
\mathcal{R}$-module, the map $\widetilde{\kappa }\circ i\circ \kappa ^{-1}:\kappa (%
\mathcal{V})\subseteq \mathcal{I}_{\mathcal{R}}(\mathcal{V})\rightarrow 
\mathcal{I}_{\mathcal{R}}(\mathcal{B})$ can be extended to a unital completely
contractive $\mathcal{R}$-map $\rho :\mathcal{I}_{\mathcal{R}}(\mathcal{V}%
)\rightarrow \mathcal{I}_{\mathcal{R}}(\mathcal{B})$ such that $\rho \circ
\kappa =\widetilde{\kappa }\circ i.$%
\begin{equation*}
\begin{tikzcd} \mathcal{V} && {\big( \mathcal{B}, \odot_{\varphi} \big) } && {\big( \mathcal{I}_{\mathcal{R}}(\mathcal{B}),
\odot_{\varphi} \big)} \\ \\ { \mathcal{I}_{\mathcal{R}}(\mathcal{V}) } \arrow["i", hook, from=1-1, to=1-3]
\arrow["\kappa"', hook, from=1-1, to=3-1] \arrow["{\widetilde{\kappa}}",
hook, from=1-3, to=1-5] \arrow["\pi"', from=1-3, to=3-1]
\arrow["{\widetilde{\pi}}"', shift left, from=1-5, to=3-1] \arrow["\rho"',
shift right=3, bend right=25, from=3-1, to=1-5] \end{tikzcd}
\end{equation*}

Then $\widetilde{\pi }\circ \rho \circ \kappa =\widetilde{\pi }\circ 
\widetilde{\kappa }\circ i=\pi \circ \kappa =\kappa $, which means that $\widetilde{%
\pi }\circ \rho \upharpoonright _{\kappa (\mathcal{V})}=$id$_{\kappa (%
\mathcal{V})}$, and since $\mathcal{I}_{\mathcal{R}}(\mathcal{V})$ is an
injective $\mathcal{R}$-envelope of $\mathcal{V}$, it follows that $%
\widetilde{\pi }\circ \rho =$id$_{\mathcal{I}_{\mathcal{R}}(\mathcal{V})}$.
On the other hand, by \cite[Corollary 5.1.2]{ER}, the maps $\widetilde{\pi }$
and $\rho $ are unital completely positive, and by Schwarz inequality (\cite[%
Corollary 5.2.2]{ER}), we have%
\begin{equation*}
y^{\ast }\odot _{\psi }y=\widetilde{\pi }\left( \rho \left( y^{\ast }\odot
_{\psi }y\right) \right) \geq \widetilde{\pi }\left( \rho \left( y\right)
^{\ast }\odot _{\varphi }\rho \left( y\right) \right) \geq \widetilde{\pi }%
(\rho (y))^{\ast }\odot _{\psi }\widetilde{\pi }\left( \rho (y)\right)
=y^{\ast }\odot _{\psi }y
\end{equation*}%
for all $y\in \mathcal{I}_{\mathcal{R}}(\mathcal{V})$, and so $$\widetilde{\pi }\left( \rho \left( y\right)
^{\ast }\odot _{\varphi }\rho \left( y\right) \right)=y^{\ast }\odot _{\psi }y=\widetilde{\pi }(\rho (y))^{\ast}\odot _{\psi }\widetilde{%
\pi }\left( \rho (y)\right).$$  Similarly, we get
that 
\begin{equation*}
\widetilde{\pi }\left( \rho (y)\odot _{\varphi }\rho (y)^{\ast }\right)
=y\odot _{\psi }y^{\ast }=\widetilde{\pi }(\rho (y))\odot _{\psi }\widetilde{%
\pi }\left( \rho (y)\right) ^{\ast }
\end{equation*}%
for all $y\in \mathcal{I}_{\mathcal{R}}(\mathcal{V})$. Then, by \cite[%
Corollary 5.2.2]{ER} , $\widetilde{\pi }\upharpoonright _{C^{\ast }\left(
\rho (\mathcal{I}_{\mathcal{R}}(\mathcal{V}))\right) }:C^{\ast }\left( \rho (%
\mathcal{I}_{\mathcal{R}}(\mathcal{V}))\right) \rightarrow \mathcal{I}_{%
\mathcal{R}}(\mathcal{V})$ is a unital $\mathcal{R}-C^{\ast }$-morphism,
where $C^{\ast }\left( \rho (\mathcal{I}_{\mathcal{R}}(\mathcal{V}))\right) $
is the unital $C^{\ast }$-subalgebra of $\mathcal{I}_{\mathcal{R}}(\mathcal{B%
})$ generated by $\rho (\mathcal{I}_{\mathcal{R}}(\mathcal{V}))$.

 On the other hand, since $\mathcal{B}$ $=C^{\ast }(\mathcal{R}i(\mathcal{V}))$ and $\widetilde{\kappa }
$ is a unital completely isometric $\mathcal{R}-C^{\ast }$-morphism (see Remark \ref{Rem1})), we have
\begin{equation*}
\widetilde{\kappa }(\mathcal{B})=\widetilde{\kappa }(C^{\ast }(\mathcal{R}i(%
\mathcal{V}))=C^{\ast }\left( \mathcal{R}\rho (\kappa (\mathcal{V}))\right)
=C^{\ast }\left( \rho (\mathcal{R}\kappa (\mathcal{V}))\right) \subseteq
C^{\ast }\left( \rho (\mathcal{I}_{\mathcal{R}}(\mathcal{V}))\right) .
\end{equation*}%
From above and since $\pi =\widetilde{\pi }\circ \widetilde{\kappa }$, it follows that $%
\pi $ is a unital $\mathcal{R}-C^{\ast }$-morphism from $\mathcal{B}$ to $%
\mathcal{I}_{\mathcal{R}}(\mathcal{V})$. Moreover, 
\begin{equation*}
\pi (\mathcal{B})=\overline{\pi (\mathcal{B})}=\overline{\pi \left( C^{\ast
}(\mathcal{R}i(\mathcal{V}))\right) }=C^{\ast }\left( \mathcal{R}\kappa (%
\mathcal{V})\right) =C_{e\mathcal{R}}^{\ast }(\mathcal{V}).
\end{equation*}%
Thus, we have showed that there exists a unital surjective $\mathcal{R}%
-C^{\ast }$-morphism $\pi :\mathcal{B}\rightarrow C_{e\mathcal{R}}^{\ast }(%
\mathcal{V})$ such that $\pi \circ i=\kappa $.

The uniqueness of the unital surjective $\mathcal{R}-C^{\ast }$-morphism $%
\pi $ follows from the definition of the unital  $C^{\ast }$-algebras 
$\mathcal{B}$ and $C_{e\mathcal{R}}^{\ast }(\mathcal{V})$ and the relation $%
\pi \circ i=\kappa .$

Suppose that $\left( \mathcal{C},j\right) $ is another $\mathcal{R}%
-C^{\ast }$-envelope of $\mathcal{V}.$ Then, by the universal property,
there exist the unital surjective $\mathcal{R}-C^{\ast }$-morphisms $\pi
_{1}:\ \mathcal{C}\rightarrow C_{e\mathcal{R}}^{\ast }(\mathcal{V})$ such
that $\pi _{1}\circ j=\kappa $ and $\pi _{2}:\ C_{e\mathcal{R}}^{\ast }(%
\mathcal{V})\rightarrow \mathcal{C}$ such that $\pi _{2}\circ \kappa=j$, respectively. Then $%
\pi _{1}\circ \pi _{2}\circ \kappa=\kappa$ and $\pi _{2}\circ \pi _{1}\circ j$ $%
= $ $j.$  

\[\begin{tikzcd}
	{\mathcal{V}} && {C_{e\mathcal{R}}^{\ast }(\mathcal{V})} \\
	{\mathcal{C}}
	\arrow["\kappa", from=1-1, to=1-3]
	\arrow["j"', from=1-1, to=2-1]
	\arrow["{\pi_{2}}", shift left, dashed, from=1-3, to=2-1]
	\arrow["{\pi_{1}}", shift left, dashed, from=2-1, to=1-3]
\end{tikzcd}\]

From these relations and the definitions of $C_{e\mathcal{R}}^{\ast
}(\mathcal{V})$ and $\mathcal{C}$, and taking into account that $\pi _{1}$
and $\pi _{2}$ are $\mathcal{R}-C^{\ast }$-morphisms, we conclude that $\pi
_{1}$ is an $\mathcal{R}-C^{\ast }$-isomorphism.

\end{proof}

\section{The $\mathcal{R}-C^{\ast }$ -envelope of a local operator space}

\subsection{The injective $\mathcal{R}$-envelope of a local operator space}

Let $\left\{ \mathcal{H};\mathcal{E};\mathcal{D}_{\mathcal{E}}\right\} $ be
a commutative quantum domain, where $\mathcal{E}:=\left\{ \mathcal{H}_{\iota
}:\iota \in \Omega \right\} $ is an upward filtered family of closed
subspaces such that the union space $\mathcal{D}_{\mathcal{E}%
}:=\bigcup\limits_{\iota \in \Omega }\mathcal{H}_{\iota }$ is dense in $%
\mathcal{H}$. Note that the quantized family $\mathcal{E}:=\left\{ \mathcal{H%
}_{\iota }:\iota \in \Omega \right\} $ determines an upward filtered,
mutually commuting family of projections $\left\{ P_{\iota }:\iota \in
\Omega \right\} $ in $B(\mathcal{H})$, where $P_{\iota }$ denotes the
orthogonal projection of $\mathcal{H}$ onto the closed subspace $\mathcal{H}%
_{\iota }$. Then 
\begin{equation*}
\{\{P_{\iota }\}_{\iota \in \Omega }\}^{\prime }:=\left\{ T\in B(\mathcal{H}%
)\mid TP_{\iota }=P_{\iota }T,\ (\forall )\ \iota \in \Omega \right\}
=b\left( C_{\mathcal{E}}^{\ast }(\mathcal{D})\right)
\end{equation*}%
is a unital injective von Neumann algebra on $\mathcal{H}$ (see \cite[IV.
2.2.7]{Bla}). Let $\mathcal{R}$ be the commutative subring of $B(\mathcal{H}%
) $ generated by $\{P_{\iota }\}_{\iota \in \Omega }\cup \{1_{\mathcal{H}}\}$%
. A unital local subspace $\mathcal{V}\subseteq C^{\ast }(\mathcal{D}_{%
\mathcal{E}})$ satisfying $\mathcal{R}\mathcal{V}\subseteq \mathcal{V}$ (or equivalently, $\mathcal{V}
\mathcal{R}\subseteq \mathcal{V}$) is called an \textit{(algebraic) local }$\mathcal{R}
$\textit{-module}.

Let $\mathcal{V}$ and $\mathcal{W}$ be two local $\mathcal{R}$-modules in $%
C^{\ast }(\mathcal{D}_{\mathcal{E}})$. A linear map $\varphi :\mathcal{%
V\rightarrow W}$ is called an $\mathcal{R}$\textit{-map} if $%
\varphi \left( P_{\iota }T\right) =P_{\iota }\varphi (T)$, for all $\iota \in
\Omega $ and $T\in \mathcal{V}.$

A unital subspace $\mathcal{V}\subseteq C^{\ast }(\mathcal{D}_{\mathcal{E}})$
is called a \textit{unital local operator space. }A unital local operator
space $\mathcal{V}\subseteq C^{\ast }(\mathcal{D}_{\mathcal{E}})$ is called
an \textit{injective local }$\mathcal{R}$\textit{-module} if $\mathcal{V}$
is injective and satisfies $\mathcal{RV\subseteq V}$ (see \cite{DD}). Note
that $\mathcal{V}$ is an injective local $\mathcal{R}$-module if and only if
it is the range of a unital local completely contractive $\mathcal{R}$%
-module projection $\varphi :C^{\ast }(\mathcal{D}_{\mathcal{E}})\rightarrow
C^{\ast }(\mathcal{D}_{\mathcal{E}})$ (\cite[Proposition 4.1]{DD}).

Dosi \cite[Proposition 4.1]{DD} shows that if $C^{\ast }(\mathcal{D}_{%
\mathcal{E}})$ is injective, then a unital local operator space $\mathcal{V}%
\subseteq C^{\ast }(\mathcal{D}_{\mathcal{E}})$ is an injective local $%
\mathcal{R}$-module if and only if $b(\mathcal{V})\subseteq b\left( C^{\ast
}(\mathcal{D}_{\mathcal{E}})\right) $ is an injective\textit{\ }$\mathcal{R}$%
-module and $\overline{b(\mathcal{V})}=\mathcal{V}$, the closure of $b(\mathcal{V})$ in $C^{\ast }(\mathcal{D}_{\mathcal{E}})$. He defines the 
\textit{injective }$\mathcal{R}$\textit{-envelope} of a unital local
operator space $\mathcal{V}\subseteq C^{\ast }(\mathcal{D}_{\mathcal{E}}),$
with $C^{\ast }(\mathcal{D}_{\mathcal{E}})$ an injective local operator
space, as the closure of the injective $\mathcal{R}$\textit{-}envelope for
the unital operator space $b(\mathcal{V})\subseteq b\left( C^{\ast }(%
\mathcal{D}_{\mathcal{E}})\right),$ namely $\mathcal{I}_{\mathcal{R}}(\mathcal{V})=%
\overline{\mathcal{I}_{\mathcal{R}}(b(\mathcal{V}))}$.

In this section, by analogy with the normed case and following the
construction of Hamana-Ruan (see, for instance, \cite{H}, \cite{Ble}, and \cite{Pau}), we define the injective $\mathcal{R}$-envelope
for a unital local operator space $\mathcal{V}\subseteq C_{\mathcal{E}}^{\ast
}(\mathcal{D})$ as a local\textit{\ }$\mathcal{R}$-rigid extension $\left( 
\mathcal{W},i\right) $ of $\mathcal{V}$, where $\mathcal{W}$
is an injective $\mathcal{R}$-module, and show that this definition coincides
with the one given by Dosi \cite{DD} (see Proposition \ref{Prop. def-equiv}%
).
\medskip

Suppose that $C^{\ast }(\mathcal{D}_{\mathcal{E}})$ is an injective space
and $\mathcal{V}\subseteq C^{\ast }(\mathcal{D}_{\mathcal{E}})$ is a unital
local operator space.

\begin{definition}
A \textit{local }$\mathcal{R}$\textit{-extension} of $\mathcal{V}$ is a
local $\mathcal{R}$-module $\mathcal{W}\subseteq C^{\ast }(\mathcal{D}_{%
\mathcal{E}})$ together with a unital local completely isometric map $i:%
\mathcal{V\rightarrow W}$.
\end{definition}

\begin{definition}
We say that a local $\mathcal{R}$-extension $\left( \mathcal{W},i\right) $
of $\mathcal{V}$ is a \textit{local }$\mathcal{R}$\textit{-rigid extension}
of $\mathcal{V}$ if id$_{\mathcal{W}}$ is the only unital local completely
contractive $\mathcal{R}$-map that extends the identity map on $\mathcal{V}$.
\end{definition}

We denote by $B_{1}^{\mathcal{R}}\left( C^{\ast }(\mathcal{D}_{\mathcal{E}%
})\right) $ \textit{the set of all unital local completely contractive }$%
\mathcal{R}$\textit{-maps }$\varphi :C^{\ast }(\mathcal{D}_{\mathcal{E}%
})\rightarrow C^{\ast }(\mathcal{D}_{\mathcal{E}}).$

\begin{remark}
\label{R. 3.3.}We note that the following statements hold.

\begin{enumerate}
\item If $\varphi \in B_{1}^{\mathcal{R}}\left( C^{\ast }(\mathcal{D}_{%
\mathcal{E}})\right) $, then $\varphi (TP_{\iota })=\varphi (T)P_{\iota }$, for all $\iota \in \Omega $, and consequently,  
\begin{equation*}
\Vert \varphi (T)\Vert _{\iota }=\Vert \varphi (T)P_{\iota }\Vert =\Vert
\varphi (TP_{\iota })\Vert \leq \Vert TP_{\iota }\Vert =\Vert T\Vert _{\iota
}
\end{equation*}%
for all $T\in C^{\ast }(\mathcal{D}_{\mathcal{E}})\ $ and all $\iota \in
\Omega .$

\item If $\varphi :C^{\ast }(\mathcal{D}_{\mathcal{E}})\rightarrow C^{\ast }(%
\mathcal{D}_{\mathcal{E}})$ is a unital local completely contractive map,
then $\varphi \left( b\left( C^{\ast }(\mathcal{D}_{\mathcal{E}})\right)
\right) \subseteq b\left( C^{\ast }(\mathcal{D}_{\mathcal{E}})\right) $,
since $\Vert \varphi (T)\Vert _{\iota }\leq \Vert T\Vert $ for all $T\in
b\left( C^{\ast }(\mathcal{D}_{\mathcal{E}})\right) $ and all $\iota \in
\Omega .$
\end{enumerate}
\end{remark}

Let $\varphi \in B_{1}^{\mathcal{R}}\left( C^{\ast }(\mathcal{D}_{\mathcal{E}%
})\right) $. Then the map $\varphi \upharpoonright _{b\left( C^{\ast }(%
\mathcal{D}_{\mathcal{E}})\right) }:b\left( C^{\ast }(\mathcal{D}_{\mathcal{E%
}})\right) \rightarrow b\left( C^{\ast }(\mathcal{D}_{\mathcal{E}})\right) $
is a unital completely contractive $\mathcal{R}$-map. Consequently, $%
\varphi \upharpoonright _{b\left( C^{\ast }(\mathcal{D}_{\mathcal{E}%
})\right) }\in B_{1}^{\mathcal{R}}\left( b\left( C^{\ast }(\mathcal{D}_{%
\mathcal{E}})\right) \right) .$
Thus, there exists a map $\Phi :B_{1}^{\mathcal{R}}\left( C^{\ast }(\mathcal{%
D}_{\mathcal{E}})\right) \rightarrow B_{1}^{\mathcal{R}}\left( b\left(
C^{\ast }(\mathcal{D}_{\mathcal{E}})\right) \right) $ defined by 
\begin{equation*}
\Phi (\varphi )=\varphi \upharpoonright _{b\left( C^{\ast }(\mathcal{D}_{%
\mathcal{E}})\right) }\text{.}
\end{equation*}

\begin{lemma}
\label{Lemma 1} The map $\Phi $ defined above establishes a bijective
correspondence between $B_{1}^{\mathcal{R}}\left( C^{\ast }(\mathcal{D}_{%
\mathcal{E}})\right) $ and $B_{1}^{\mathcal{R}}\left( b\left( C^{\ast }(%
\mathcal{D}_{\mathcal{E}})\right) \right) $.
\end{lemma}

\begin{proof}
Let $\varphi _{1},\varphi _{2}\in B_{1}^{\mathcal{R}}\left( C^{\ast }(%
\mathcal{D}_{\mathcal{E}})\right) $ be such that $\Phi (\varphi _{1})=\Phi
(\varphi _{2})$. It follows that $\varphi _{1}\upharpoonright _{b\left(
C^{\ast }(\mathcal{D}_{\mathcal{E}})\right) }=\varphi _{2}\upharpoonright
_{b\left( C^{\ast }(\mathcal{D}_{\mathcal{E}})\right) }$. Since both $%
\varphi _{1}\upharpoonright _{b\left( C^{\ast }(\mathcal{D}_{\mathcal{E}%
})\right) }$ and $\varphi _{2}\upharpoonright _{b\left( C^{\ast }(\mathcal{D}%
_{\mathcal{E}})\right) }$ are continuous with respect to the family of
seminorms $\left\{ \Vert \cdot \Vert _{\iota }\upharpoonright _{b\left(
C^{\ast }(\mathcal{D}_{\mathcal{E}})\right) }\right\} _{\iota \in \Omega }$,
and taking into account Remark \ref{R. 3.3.} (2) as well as the fact that $%
b\left( C^{\ast }(\mathcal{D}_{\mathcal{E}})\right) $ is dense in $C^{\ast }(%
\mathcal{D}_{\mathcal{E}})$, it follows that $\varphi _{1}=\varphi _{2}.$
Therefore, the map $\Phi :\ \varphi \mapsto \varphi \upharpoonright
_{b\left( C^{\ast }(\mathcal{D}_{\mathcal{E}})\right) }$ is injective.

Let $\psi \in B_{1}^{\mathcal{R}}\left( b\left( C^{\ast }(\mathcal{D}_{%
\mathcal{E}})\right) \right) $. Then $\psi $ is continuous with respect to
the family of seminorms $\left\{ \Vert \cdot \Vert _{\iota }\upharpoonright
_{b\left( C^{\ast }(\mathcal{D}_{\mathcal{E}})\right) }\right\} _{\iota \in
\Omega }$, since it is an $\mathcal{R}$-map. Hence, there exists a unital
local completely contractive $\mathcal{R}$-map $\widetilde{\psi }:C^{\ast }(%
\mathcal{D}_{\mathcal{E}})\rightarrow C^{\ast }(\mathcal{D}_{\mathcal{E}})$
such that $\widetilde{\psi }\upharpoonright _{b\left( C^{\ast }(\mathcal{D}_{%
\mathcal{E}})\right) }=\varphi .$ Consequently, we have $\Phi (\widetilde{%
\psi })=\psi .$ Therefore, the map $\Phi :\ \varphi \mapsto \varphi
\upharpoonright _{b\left( C_{\mathcal{E}}^{\ast }(\mathcal{D})\right) }$ is
surjective.
\end{proof}

Let
\begin{equation*}
\mathscr{E}_{\mathcal{V}}:=\left\{ \varphi \in B_{1}^{\mathcal{R}}\left(
C^{\ast }(\mathcal{D}_{\mathcal{E}})\right) \mid \varphi \upharpoonright _{%
\mathcal{V}}=\text{id}_{\mathcal{V}}\right\} .
\end{equation*}%
For $\varphi ,\psi \in \mathscr{E}_{\mathcal{V}}$, we define the relation $%
\varphi \prec \psi $ if and only if 
\begin{equation*}
\Vert \varphi (T)\Vert _{\iota }\leq \Vert \psi (T)\Vert _{\iota }
\end{equation*}%
for all $\ T\in C^{\ast }(\mathcal{D}_{\mathcal{E}})$ and all $\iota \in
\Omega .$ This relation is reflexive and transitive on $\mathscr{E}_{\mathcal{V}%
}$. We write $\varphi \approx \psi $ if and only if $\varphi \prec \psi $
and $\psi \prec \varphi .$ An element $\varphi \in \mathscr{E}_{\mathcal{V}}$
is called \textit{minimal} if $\psi \prec \varphi $ implies $\psi \approx
\varphi .$ We denote by $\min \mathscr{E}_{\mathcal{V}}$ the set of all minimal
elements of $\mathscr{E}_{V}$.

\begin{lemma}\label{Lemma 2}
Let $\mathcal{V}\subseteq C^{\ast }(\mathcal{D}_{\mathcal{E}})$ be a unital local operator space. Then the set $\min \mathscr{E}_{\mathcal{V}}$ is non-empty and
consists of projections in $\mathscr{E}_{\mathcal{V}}$. Moreover, the
restriction of the map $\Phi $ to $\min \mathscr{E}_{\mathcal{V}}$ establishes
a bijective correspondence between $\min \mathscr{E}_{\mathcal{V}}$ and $\min 
\mathscr{E}_{b(\mathcal{V})}$.
\end{lemma}

\begin{proof}
Clearly, $\Phi \left( \mathscr{E}_{\mathcal{V}}\right) =\mathscr{E}_{b(\mathcal{V}%
)}$, and $\varphi $ is a projection in $\mathscr{E}_{\mathcal{V}}$ if and only
if $\varphi \upharpoonright _{b\left( C^{\ast }(\mathcal{D}_{\mathcal{E}%
})\right) }$ is a projection in $b(\mathcal{V}).$

Let $\varphi \in \min \mathscr{E}_{b(\mathcal{V})}$, since $\min \mathscr{E}_{b(%
\mathcal{V})}\neq \emptyset $ by \cite[Lemma 3.3]{DD}. By Lemma \ref{Lemma 1}%
, there exists $\widetilde{\varphi }\in B_{1}^{\mathcal{R}}\left( C^{\ast }(%
\mathcal{D}_{\mathcal{E}})\right) $ such that $\widetilde{\varphi }%
\upharpoonright _{b\left( C^{\ast }(\mathcal{D}_{\mathcal{E}})\right)
}=\varphi .$ Clearly, $\widetilde{\varphi }\upharpoonright _{\mathcal{V}}=$id%
$_{\mathcal{V}}$. Hence, $\widetilde{\varphi }\in \mathscr{E}_{\mathcal{V}}.$
Now, let $\psi \in \mathscr{E}_{\mathcal{V}}$ be such that $\psi \prec 
\widetilde{\varphi }.$ Then, $\psi \upharpoonright _{b\left( C^{\ast }(%
\mathcal{D}_{\mathcal{E}})\right) }\in \mathscr{E}_{b(\mathcal{V})}$ and $\psi
\upharpoonright _{b\left( C^{\ast }(\mathcal{D}_{\mathcal{E}})\right)
}\preceq \varphi $. Since $\varphi $ is minimal, it follows that $\psi
\upharpoonright _{b\left( C^{\ast }(\mathcal{D}_{\mathcal{E}})\right) }\sim
\varphi $, and by continuity, $\psi \approx \widetilde{\varphi }.$
Therefore, $\widetilde{\varphi }\in \min \mathscr{E}_{\mathcal{V}}.$
Consequently, $\min \mathscr{E}_{b(V)}\subseteq \Phi (\min \mathscr{E}_{V}).$

Let $\varphi \in \min \mathscr{E}_{\mathcal{V}}$. Then, $\varphi
\upharpoonright _{b\left( C^{\ast }(\mathcal{D}_{\mathcal{E}})\right) }\in
\mathscr{E}_{b(\mathcal{V})}$. Let $\psi \in \mathscr{E}_{b(\mathcal{V})}$ be such that $\psi
\preceq \varphi \upharpoonright _{b\left( C^{\ast }(\mathcal{D}_{\mathcal{E}%
})\right) }$. Then, there exists $\widetilde{\psi }\in \mathscr{E}_{\mathcal{V}%
} $ such that $\widetilde{\psi }\upharpoonright _{b\left( C^{\ast }(\mathcal{%
D}_{\mathcal{E}})\right) }=\psi $. Clearly, $\widetilde{\psi }\prec \varphi $%
. Since $\varphi $ is minimal, it follows that $\widetilde{\psi }\approx
\varphi $. Consequently, $\psi \sim \varphi \upharpoonright _{b\left(
C^{\ast }(\mathcal{D}_{\mathcal{E}})\right) }$. Therefore, $\Phi (\min \mathscr{E}_{V})\subseteq \min \mathscr{E}_{b(V)}.$
\end{proof}

\begin{proposition}\label{P1}
Let $\mathcal{V}\subseteq C^{\ast }(\mathcal{D}_{\mathcal{E}})$ be a unital local operator space.
Suppose that $\varphi \in \min \mathscr{E}_{\mathcal{V}}\neq \emptyset $.
Then $\left( \text{Im}(\varphi ),i\right) $ is a local $\mathcal{R}$-rigid
extension of $\mathcal{V}$.
\end{proposition}

\begin{proof}
By \cite[Proposition 4.1]{DD}, Im$(\varphi )$ is an injective local $%
\mathcal{R}$-module. Then, there exists a unital local completely
contractive $\mathcal{R}$-map $\psi :$Im$(\varphi )\rightarrow $Im$(\varphi
) $ such that $\psi \upharpoonright _{\mathcal{V}}=$id$_{\mathcal{V}}$. Let $%
\widetilde{\varphi }=\psi \circ \varphi $. Clearly, $\widetilde{\varphi }\in
B_{1}^{\mathcal{R}}\left( C^{\ast }(\mathcal{D}_{\mathcal{E}})\right) ,\ 
\widetilde{\varphi }\upharpoonright _{\mathcal{V}}=$id$_{\mathcal{V}}$, and
Im$(\widetilde{\varphi })\subseteq $Im$(\varphi ).$ Hence, $\widetilde{%
\varphi }\in \mathscr{E}_{\mathcal{V}}$ and $\Vert \widetilde{\varphi }(T)\Vert
_{\iota }\leq \Vert \varphi (T)\Vert _{\iota }$ for all $\ T\in C^{\ast }(%
\mathcal{D}_{\mathcal{E}})$ and all $\iota \in \Omega .$ Therefore, $%
\widetilde{\varphi }\preceq \varphi $, and since $\varphi $ is minimal, it
follows that $\widetilde{\varphi }\approx \varphi $. Consequently, $%
\widetilde{\varphi }\in \min \mathscr{E}_{\mathcal{V}}$. Now, for any $T\in C^{\ast
}(\mathcal{D}_{\mathcal{E}})$, we have 
\begin{align*}
\Vert \widetilde{\varphi }(T)-\varphi (T)\Vert _{\iota }& =\Vert \left(
\varphi \circ \widetilde{\varphi }\right) (T)-\varphi (T)\Vert _{\iota } \\
& =\Vert \varphi \left( \widetilde{\varphi }(T)-T\right) \Vert _{\iota } \\
& =\Vert \widetilde{\varphi }\left( \widetilde{\varphi }(T)-T\right) \Vert
_{\iota } \\
& =0
\end{align*}%
for all $\iota \in \Omega .$ Therefore, $\widetilde{\varphi }=\varphi ,$ and so, $\psi =$id$_{\text{Im}(\varphi )}$. This completes the proof.
\end{proof}

\begin{definition}
\label{Def-INJ-Env} Let $\mathcal{V}\subseteq C^{\ast }(\mathcal{D}_{%
\mathcal{E}})$ be a unital local operator space. We define the \textit{%
injective }$\mathcal{R}$\textit{-envelope} of $\mathcal{V}$ as a local $%
\mathcal{R}$-rigid extension $\left( \mathcal{W},i\right) $ of $\mathcal{V}$
with $\mathcal{W}$ an injective $\mathcal{R}$ -module.
\end{definition}

Let $\mathcal{V}\subseteq C^{\ast }(\mathcal{D}_{\mathcal{E}})$ and $%
\mathcal{W}\subseteq C^{\ast }(\mathcal{D}_{\mathcal{E}})$ be two unital
local $\mathcal{R}$-modules. We say that a unital linear map $\varphi :%
\mathcal{V\rightarrow W}$ is a \textit{unital local complete order }$%
\mathcal{R}$\textit{-isomorphism} if $\varphi $ is a local complete order
isomorphism and an $\mathcal{R}$-map.

\begin{proposition}
\label{Coro} Any unital local operator space $\mathcal{V}\subseteq C^{\ast }(%
\mathcal{D}_{\mathcal{E}})$ has an injective $\mathcal{R}$-envelope, which
is unique up to a unital local complete order $\mathcal{R}$-isomorphism.
\end{proposition}

\begin{proof}
By Proposition \ref{P1} and \cite[Proposition 4.1]{DD}, it follows that $%
\left( \text{Im}(\varphi ),i\right) $ is an injective $\mathcal{R}$-envelope
for $\mathcal{V}$, where $\varphi \in \min \mathscr{E}_{\mathcal{V}}.$

For uniqueness, let $\left( \mathcal{W}_{1},i_{1}\right) $ and $\left( 
\mathcal{W}_{2},i_{2}\right) $ be two injective $\mathcal{R}$-envelopes of $%
\mathcal{V}$. By the injectivity of $\mathcal{W}_{1}$ and $\mathcal{W}_{2}$,
there exist unital local completely contractive $\mathcal{R}$-extensions of
the identity map id$_{\mathcal{V}}:\mathcal{V\rightarrow V}$, denoted by $%
\Phi $ and $\Psi $, respectively. 
\begin{equation*}
\ \begin{tikzcd} {\mathcal{W}_{1}} && {\mathcal{W}_{2}} && {\mathcal{W}_{1}} \\ \mathcal{V} && \mathcal{V} && \mathcal{V} \arrow["\Phi",
from=1-1, to=1-3] \arrow["\Psi", from=1-3, to=1-5] \arrow["{i_{1}}", hook,
from=2-1, to=1-1] \arrow["{id_{\mathcal{V}}}", from=2-1, to=2-3] \arrow["{i_{2}}",
hook, from=2-3, to=1-3] \arrow["{id_{\mathcal{V}}}", from=2-3, to=2-5]
\arrow["{i_{1}}", hook, from=2-5, to=1-5] \end{tikzcd}
\end{equation*}

Since $\left( \Psi \circ \Phi \right) \upharpoonright _{\mathcal{V}}=$id$_{%
\mathcal{V}}$, by the $\mathcal{R}$-rigidity property it follows that $\Psi
\circ \Phi =$id$_{\mathcal{W}_{1}}$. Similarly, we have $\Phi \circ \Psi =$id%
$_{\mathcal{W}_{2}}.$ Hence, $\Phi $ is a unital local complete order $%
\mathcal{R}$-isomorphism.
\end{proof}

\begin{proposition}
\label{Prop. def-equiv} Let $\mathcal{V}\subseteq C^{\ast }(\mathcal{D}_{%
\mathcal{E}})$ be a unital local operator space such that $\overline{b(%
\mathcal{V})}=\mathcal{V}$. Then $b\left( \mathcal{I}_{\mathcal{R}}(\mathcal{%
V})\right) =\mathcal{I}_{\mathcal{R}}\left( b(\mathcal{V})\right) $, up to a
unital complete order $\mathcal{R}$-isomorphism.
\end{proposition}

\begin{proof}
By Corollary \ref{Coro}, there exists $\varphi \in \min \mathscr{E}_{\mathcal{V}}$
such that Im$(\varphi )=\mathcal{I}_{\mathcal{R}}(\mathcal{V})$. Moreover,
by Lemma \ref{Lemma 2} and \cite[page 1026]{DD}, we obtain Im$\left( \varphi
\upharpoonright _{b\left( C^{\ast }(\mathcal{D}_{\mathcal{E}})\right)
}\right) =\mathcal{I}_{\mathcal{R}}\left( b(\mathcal{V})\right) $, up to a
unital complete order $\mathcal{R}$-isomorphism. Clearly, we have $\varphi
\left( b\left( C^{\ast }(\mathcal{D}_{\mathcal{E}})\right) \right) \subseteq
b\left( \varphi \left( C^{\ast }(\mathcal{D}_{\mathcal{E}})\right) \right) $%
. Consequently, it follows that $\mathcal{I}_{\mathcal{R}}\left( b(\mathcal{V%
})\right) \subseteq b\left( \mathcal{I}_{\mathcal{R}}(\mathcal{V})\right) .$

Now we prove that $b\left( \varphi \left( C^{\ast }(\mathcal{D}_{\mathcal{E}%
})\right) \right) \subseteq \varphi \left( b\left( C^{\ast }(\mathcal{D}_{%
\mathcal{E}})\right) \right) .$ For this, let $T\in b\left( \varphi \left(
C_{\mathcal{E}}^{\ast }(\mathcal{D})\right) \right) ,$ that is, $T=\varphi
(S),\ S\in C^{\ast }(\mathcal{D}_{\mathcal{E}})$ and $\Vert T\Vert <\infty $%
. By applying $\varphi $ to the last equality, we obtain  
\begin{equation*}
\varphi (T)=\varphi \left( \varphi (S)\right) =\varphi (S)=T.
\end{equation*}%
Hence, $T\in \varphi \left( b\left( C^{\ast }(\mathcal{D}_{\mathcal{E}%
})\right) \right) $. This completes the proof.
\end{proof}

Let $\{\mathcal{H};\mathcal{E}=\{\mathcal{H}_{n}\}_{n\geq 1};\mathcal{D}_{%
\mathcal{E}}\}$ be a Fr\'{e}chet quantized domain in the Hilbert space $%
\mathcal{H}$. For each $\xi ,\eta \in \mathcal{D}_{\mathcal{E}}$, the rank-one operator $\theta _{\xi ,\eta }:\mathcal{D}_{\mathcal{E}}\rightarrow 
\mathcal{D}_{\mathcal{E}},\theta _{\xi ,\eta }\left( \zeta \right) =\xi
\left\langle \eta ,\zeta \right\rangle $ is an element of $C^{\ast }(%
\mathcal{D}_{\mathcal{E}})$. The closure of the linear space generated by
the rank-one operators is a closed two sided $\ast $-ideal of $C^{\ast }(%
\mathcal{D}_{\mathcal{E}})$, denoted by $K^{\ast }(\mathcal{D}_{\mathcal{E}%
}) $ (for more details, see \cite{KP} and \cite{I}).

Moreover, for each $n\geq 1,$ $C^{\ast }(\mathcal{D}_{\mathcal{E}})\mathcal{P%
}_{n}$ and $K^{\ast }(\mathcal{D}_{\mathcal{E}})\mathcal{P}_{n}$ are $%
C^{\ast }$-subalgebras of $B(\mathcal{H}_{n})$, and they are isomorphic to $%
\left( C^{\ast }(\mathcal{D}_{\mathcal{E}})\right) _{n}$ and $%
\left( K^{\ast }(\mathcal{D}_{\mathcal{E}})\right) _{n}$, respectively (see \cite{D3}).

On the other hand, for each $n\geq 1$, $\mathcal{P}_{n}$ denotes the
orthogonal projection of $\mathcal{H}$ onto $\mathcal{H}_{n}$. Then $T\in
C^{\ast }(\mathcal{D}_{\mathcal{E}})$ if and only if 
\begin{equation*}
T=\sum\limits_{n=1}^{\infty }\left( \text{id}_{\mathcal{H}}-\mathcal{P}%
_{n-1}\right) \mathcal{P}_{n}T\left( \text{id}_{\mathcal{H}}-\mathcal{P}%
_{n-1}\right) \mathcal{P}_{n}\upharpoonright _{\mathcal{D}_{\mathcal{E}}}
\end{equation*}%
where $\mathcal{P}_{0}=0$ (see \cite[Proposition 4.2]{D}). Therefore, for each $%
n\geq 1$, 
\begin{equation*}
\left( C^{\ast }(\mathcal{D}_{\mathcal{E}})\right) _{n}=B(\mathcal{H}%
_{1})\oplus B(\left( \mathcal{H}_{1}\right) ^{\bot }\cap \mathcal{H}%
_{2})\oplus \cdot \cdot \cdot \oplus B(\left( \mathcal{H}_{n-1}\right)
^{\bot }\cap \mathcal{H}_{n})
\end{equation*}%
and 
\begin{equation*}
\left( K^{\ast }(\mathcal{D}_{\mathcal{E}})\right) _{n}=K(\mathcal{H}%
_{1})\oplus K(\left( \mathcal{H}_{1}\right) ^{\bot }\cap \mathcal{H}%
_{2})\oplus \cdot \cdot \cdot \oplus K(\left( \mathcal{H}_{n-1}\right)
^{\bot }\cap \mathcal{H}_{n}).
\end{equation*}%
From the above relations, and taking into account that $B(\mathcal{H})$ is
the injective envelope for $K(\mathcal{H})$, we conclude that the
identity map on $\left( C^{\ast }(\mathcal{D}_{\mathcal{E}})\right) _{n}$ is
the unique completely contractive map which extends the identity map on $%
\left( K^{\ast }(\mathcal{D}_{\mathcal{E}})\right) _{n}.$

\begin{example}
Let $\{\mathcal{H};\mathcal{E}=\{\mathcal{H}_{n}\}_{n\geq 1};\mathcal{D}_{%
\mathcal{E}}\}$ be a Fr\'{e}chet quantized domain in the Hilbert space $%
\mathcal{H}$, and $\mathcal{V}\subseteq C^{\ast }(\mathcal{D}_{\mathcal{E}})$
be a unital local operator space. If $K^{\ast }(\mathcal{D}_{\mathcal{E}%
})\subseteq \mathcal{V}$, then $\mathcal{I}_{\mathcal{R}}(\mathcal{V})=$ $%
C^{\ast }(\mathcal{D}_{\mathcal{E}})$. Indeed, $C^{\ast }(\mathcal{D}_{%
\mathcal{E}})$ is injective and an $\mathcal{R}$-module. If $\varphi :C^{\ast }(%
\mathcal{D}_{\mathcal{E}})\rightarrow C^{\ast }(\mathcal{D}_{\mathcal{E}})$
is a unital local completely contractive $\mathcal{R}$-map such that $%
\varphi \upharpoonright _{\mathcal{V}}=$id$_{\mathcal{V}},$ then $\varphi
\upharpoonright _{K^{\ast }(\mathcal{D}_{\mathcal{E}})}=$id$_{K^{\ast }(%
\mathcal{D}_{\mathcal{E}})}.$ Since $\varphi $ is an $\mathcal{R}$-map,
for each $n\geq 1$ the map $\varphi _{n}:\left( C^{\ast }(\mathcal{D}_{%
\mathcal{E}})\right) _{n}\rightarrow \left( C^{\ast }(\mathcal{D}_{\mathcal{E%
}})\right) _{n}$ given by $\varphi _{n}\left( T\upharpoonright _{\mathcal{H}%
_{n}}\right) $ $=\varphi \left( T\right) \mathcal{P}_{n}$ is completely
contractive and satisfies $\varphi _{n}\upharpoonright _{\left( K^{\ast }(%
\mathcal{D}_{\mathcal{E}})\right) _{n}}=$id$_{\left( K^{\ast }(\mathcal{D}_{%
\mathcal{E}})\right) _{n}}.$ Therefore, $\varphi _{n}=$id$_{\left( C^{\ast }(%
\mathcal{D}_{\mathcal{E}})\right) _{n}}$ for all $n\geq 1$, and then $\varphi
= $id$_{C^{\ast }(\mathcal{D}_{\mathcal{E}})}$
\end{example}

\subsection{The $\mathcal{R}-C^{\ast }$-envelope of a local operator space}

By analogy with the normed case, we introduce the notion of a local\textit{\ }$%
\mathcal{R}-C^{\ast }$-envelope of a unital local operator space, and we show
that any unital local operator space has a local\textit{\ }$\mathcal{R}%
-C^{\ast }$-envelope which is unique up to a unital local isometric $%
\mathcal{R}-\ast $-isomorphism. As in the case of the injective $\mathcal{R}$%
\textit{-}envelope for a unital local operator space, we prove that the local%
\textit{\ }$\mathcal{R}-C^{\ast }$-envelope of a unital local operator space 
$\mathcal{V}$ is the closure of the $\mathcal{R}-C^{\ast }$-envelope of 
its bounded part $b(\mathcal{V}).$

\begin{remark}
\label{Corollary} By \cite[Theorem 8.3]{D}, the range of a unital local
completely contractive projection $\varphi :C^{\ast }(\mathcal{D}_{\mathcal{E%
}})\rightarrow C^{\ast }(\mathcal{D}_{\mathcal{E}})$ has a structure of a
unital $\ast $-algebra, where the multiplication is defined by $T\odot
_{\varphi }S:=\varphi (TS)$, and the involution is inherited from $C^{\ast }(%
\mathcal{D}_{\mathcal{E}})$. In general, the family of $C^{\ast }$-seminorms 
$\{\Vert \cdot \Vert _{\iota }\}_{\iota \in \Omega }$ need not satisfy the $%
C^{\ast }$-seminorm condition with respect to this new multiplication.
However, if $\varphi $ is also an $\mathcal{R}$-map, then by Remark \ref{R.
3.3.} (1) the pair $\left( \text{Im}(\varphi ),\odot _{\varphi }\right) $
forms a unital locally $C^{\ast }$-algebra with the topology induced from $%
C^{\ast }(\mathcal{D}_{\mathcal{E}}).$
\end{remark}

Assume that $C^{\ast }(\mathcal{D}_{\mathcal{E}})$ is an injective space, 
and let $\mathcal{V}\subseteq C^{\ast }(\mathcal{D}_{\mathcal{E}})$ be a unital
local operator space.

\begin{definition}
Let $\mathcal{V}\subseteq C^{\ast }(\mathcal{D}_{\mathcal{E}})$ be a unital
local operator space. A \textit{local }$\mathcal{R}-C^{\ast }$\textit{%
-extension} of $\mathcal{V}$ is a local $\mathcal{R}$-extension $\left( 
\mathcal{B},i\right) $ with the property that $\mathcal{B}$ is a unital
locally $C^{\ast }$-algebra generated by $\mathcal{R}i(\mathcal{V})$. (Note
that, in general, the locally $C^{\ast }$-algebra structure of $\mathcal{B}$
may differ from that of $C^{\ast }(
\mathcal{D}_{\mathcal{E}})$).
\end{definition}

\begin{definition}
\label{UP-local} Let $\mathcal{V}\subseteq C^{\ast }(\mathcal{D}_{\mathcal{E}%
})$ be a unital local operator space. A \textit{local }$\mathcal{R}-C^{\ast
} $-\textit{envelope} of $\mathcal{V}$ is a local $\mathcal{R}-C^{\ast }$%
-extension $\left( \mathcal{D},\kappa \right) $ with the following universal
property: Given any local $\mathcal{R}-C^{\ast }$-extension $\left( \mathcal{%
C},i\right) $ of $V$, there exists a unique unital local contractive $%
\mathcal{R}-\ast $-morphism $\pi :\mathcal{C\rightarrow D}$ with dense range
such that $\pi \circ i=\kappa .$
\end{definition}

\begin{theorem}
\label{Th-C^*-env} Any unital local operator space $\mathcal{V}\subseteq
C^{\ast }(\mathcal{D}_{\mathcal{E}})$ has a local $\mathcal{R}-C^{\ast }$%
-envelope.
\end{theorem}

\begin{proof}
Using Lemma \ref{Lem} and Theorem \ref{Mult.Dom}, the proof of this theorem is
similar to the proof of Theorem \ref{1}.
\end{proof}

\begin{proposition}
\label{uniqueness envelopes} Let $\mathcal{V}_{1}\subseteq C_{\mathcal{E}%
}^{\ast }(\mathcal{D})$ and $\mathcal{V}_{2}\subseteq C^{\ast }(\mathcal{D}_{%
\mathcal{E}})$ be two unital local operator spaces, and let $\left( \mathcal{%
D}_{1},\kappa _{1}\right) $ and $\left( \mathcal{D}_{2},\kappa _{2}\right) $
be local $\mathcal{R}-C^{\ast }$-envelopes of $\mathcal{V}_{1}$ and $%
\mathcal{V}_{2}$, respectively. If $\theta :\mathcal{V}_{1}\rightarrow 
\mathcal{V}_{2}$ is a unital local complete order $\mathcal{R}$-isomorphism,
then there exists a unique unital local isometric $\mathcal{R}-\ast $%
-isomorphism $\pi :\mathcal{D}_{1}\rightarrow \mathcal{D}_{2}$ such that $%
\pi \circ \kappa _{1}=\kappa _{2}\circ \theta .$
\end{proposition}

\begin{proof}
Clearly, $\left( \mathcal{D}_{2},\kappa _{2}\circ \theta \right) $ and $%
\left( \mathcal{D}_{1},\kappa _{1}\circ \theta ^{-1}\right) $ are local $%
\mathcal{R}-C^{\ast }$-extensions of $\mathcal{V}_{1}$ and $\mathcal{V}_{2}$%
, respectively. Since $\left( \mathcal{D}_{1},\kappa _{1}\right) $ and $%
\left( \mathcal{D}_{2},\kappa _{2}\right) $ are local $\mathcal{R}-C^{\ast }$%
-envelopes of $\mathcal{V}_{1}$ and $\mathcal{V}_{2}$, respectively, there
exist unique local contractive $\mathcal{R}-\ast $-morphisms with dense
range, $\rho :\mathcal{D}_{2}\rightarrow \mathcal{D}_{1}$ and $\pi :\mathcal{%
D}_{1}\rightarrow \mathcal{D}_{2}$ such that $\rho \circ \kappa _{2}\circ
\theta =\kappa _{1}$\ and $\pi \circ \kappa _{1}\circ \theta ^{-1}=\kappa
_{2},$ respectively.%
\begin{equation*}
\begin{tikzcd} {\mathcal{V}_{1}} && {\mathcal{D}_{1}} \\ \\ {\mathcal{V}_{2}} && {\mathcal{D}_{2}}
\arrow["{\kappa_{1}}", from=1-1, to=1-3] \arrow["\iota"', from=1-1, to=3-1]
\arrow["\pi"', shift right=2, from=1-3, to=3-3] \arrow["{\kappa_{2}}",
from=3-1, to=3-3] \arrow["\rho"', shift right=2, from=3-3, to=1-3]
\end{tikzcd}
\end{equation*}

Then $\rho \circ \pi :\mathcal{D}_{1}\rightarrow \mathcal{D}_{1}$ and $\pi
\circ \rho :\mathcal{D}_{2}\rightarrow \mathcal{D}_{2}$ are local
contractive $\mathcal{R}-\ast $-morphisms with dense range. Moreover, 
\begin{equation*}
\rho \circ \pi \circ \kappa _{1}=\rho \circ \kappa _{2}\circ \theta =\kappa
_{1}
\end{equation*}%
and 
\begin{equation*}
\pi \circ \rho \circ \kappa _{2}=\pi \circ \kappa _{1}\circ \theta
^{-1}=\kappa _{2}\circ \theta \circ \theta ^{-1}=\kappa _{2}.
\end{equation*}%
From the above relations and taking into account that $\left( \mathcal{D}%
_{1},\kappa _{1}\right) $ and $\left( \mathcal{D}_{2},\kappa _{2}\right) $
are local $\mathcal{R}-C^{\ast }$-envelopes of $\mathcal{V}_{1}$ and $%
\mathcal{V}_{2}$, respectively, it follows that $\rho \circ \pi =$id$_{D_{1}}$
and $\pi \circ \rho =$id$_{D_{2}}$. Therefore, $\pi \ $is invertible and $%
\pi ^{-1}=\rho $. Consequently, $\pi $ is a local contractive $\mathcal{R}%
-\ast $-isomorphism. On the other hand, for each $d_{1}\in \mathcal{D}_{1}$
there exists $d_{2}\in \mathcal{D}_{2}$ such that $d_{1}=$ $\rho \left(
d_{2}\right) ,$ and then 
\begin{equation*}
p_{\iota }\left( d_{1}\right) =p_{\iota }\left( \rho \left( d_{2}\right)
\right) \leq p_{\iota }\left( d_{2}\right) =p_{\iota }\left( \pi \left( \rho
\left( d_{2}\right) \right) \right) =p_{\iota }\left( \pi \left(
d_{1}\right) \right) \leq p_{\iota }\left( d_{1}\right)
\end{equation*}%
for all $\iota \in \Omega .$ Therefore, $\pi $ is a local isometric $%
\mathcal{R}-\ast $-isomorphism, and the proposition is proved.
\end{proof}

\begin{corollary}\label{Corollary C^*-env}
Let $\mathcal{V}\subseteq C^{\ast }(\mathcal{D}_{\mathcal{E}})$ be a unital local operator space.
 Then the local $\mathcal{R}-C^{\ast }$-envelope of $%
\mathcal{V}\subseteq C^{\ast }(\mathcal{D}_{\mathcal{E}})$ is unique up to a
unital local isometric $\mathcal{R}-\ast $-isomorphism.
\end{corollary}

We use the notation $\left( C_{e\mathcal{R}}^{\ast }(\mathcal{V}),\kappa
\right) $ for any local $\mathcal{R}-C^{\ast }$-envelope of a unital local
operator space $\mathcal{V}$.

\begin{remark}
\label{Remarkk C^*-env} Let $\mathcal{V}\subseteq C^{\ast }(\mathcal{D}_{%
\mathcal{E}})$ be a unital local operator space satisfying $\overline{b(%
\mathcal{V})}=\mathcal{V}$. Then $C_{e\mathcal{R}}^{\ast }(\mathcal{%
V})=\overline{C_{e\mathcal{R}}^{\ast }(b(\mathcal{V}))}$. Moreover, we have that $C_{e%
\mathcal{R}}^{\ast }(b(\mathcal{V}))\subseteq \ b\left( C_{e\mathcal{R}%
}^{\ast }(\mathcal{V})\right) $. 
\end{remark}

\section{The $\mathcal{R}$-Shilov boundary ideal for a local operator space}

In this section, we introduce the concept of the local $\mathcal{R}$-Shilov
boundary ideal for a unital local operator space, inspired by the pioneering work of Arveson \cite{Arv69}.
\medskip

Let $\mathcal{A}$ be a locally $C^{\ast }$-algebra whose topology is defined
by the family of $C^{\ast }$-seminorms $\{p_{\iota }\}_{\iota \in \Omega }$.
Suppose that $\mathcal{I}$ is a closed two-sided $\ast $-ideal of $\mathcal{A%
}$. Then the quotient $\ast $-algebra $\mathcal{A}/\mathcal{I}$ is a
topological $\ast $- algebra with respect to the family of $C^{\ast }$%
-seminorms $\{\widehat{p}_{\iota }\}_{\iota \in \Omega }$, where $\widehat{p}%
_{\iota }(a+\mathcal{I})=\inf \{p_{\iota }(a+b)\mid b\in \mathcal{I}\}$. In
general, $\mathcal{A}/\mathcal{I}$ is not complete with respect to the
family of $C^{\ast }$-seminorms $\{\widehat{p}_{\iota }\}_{\iota \in \Omega
} $ (for further details see \cite{Fr}). Note that, for each $\iota \in
\Omega $, the $C^{\ast }$-algebras $\overline{\left( \mathcal{A}/\mathcal{I}%
\right) _{\iota }}$ and $\mathcal{A}_{\iota }/\mathcal{I}_{\iota }$ are
isomorphic, where $\overline{\left( \mathcal{A}/\mathcal{I}\right) _{\iota }}
$ is the completion of the normed $\ast $-algebra $\left( \mathcal{A}/%
\mathcal{I}\right) /\ker (\widehat{p}_{\iota })$, and $\mathcal{I}_{\iota }$
is the closure of $\pi _{\iota }^{\mathcal{A}}(\mathcal{I})$ in $\mathcal{A}%
_{\iota }.$ Therefore, the locally $C^{\ast }$-algebra $\overline{\mathcal{A%
}/\mathcal{I}}$ can be identified with $\varprojlim\limits_{\iota}\mathcal{A}_{\iota }/\mathcal{I}_{\iota }$, and the canonical $%
\ast $-morphism $\sigma :\mathcal{A}\rightarrow \overline{\mathcal{A}/%
\mathcal{I}}$ is an inverse limit $\ast $-morphism, $\sigma =\varprojlim\limits_{\iota} \sigma _{\iota }$, where $\sigma _{\iota }$
is the canonical homomorphism from $\mathcal{A}_{\iota }$ onto $\mathcal{A}%
_{\iota }/\mathcal{I}_{\iota }$ for all $\iota \in \Omega .$

\begin{definition}
\label{Deff Shilov R-ideal} Let $\mathcal{V}\subseteq C^{\ast }(\mathcal{D}_{%
\mathcal{E}})$ be a unital local operator space and let $\mathcal{C}%
\subseteq C^{\ast }(\mathcal{D}_{\mathcal{E}})$ be the unital locally $%
C^{\ast }$-algebra generated $\mathcal{RV}$. A closed two-sided $\ast $%
-ideal $\mathcal{I}$ of $\mathcal{C}$ is called a \textit{local boundary }$%
\mathcal{R}$\textit{-ideal} for $\mathcal{V}$ if the canonical map $\sigma :%
\mathcal{C}\rightarrow \overline{\mathcal{C}/\mathcal{I}}$ is local
completely isometric on $\mathcal{V}$. A local boundary $\mathcal{R}$-ideal
for $\mathcal{V}$ is called \textit{the local }$\mathcal{R}$\textit{-Shilov
boundary} for $\mathcal{V}$ if it contains every other local boundary $%
\mathcal{R}$-ideal.

\end{definition}
Although the following result may be known in some form, we could not find a
precise reference within the framework of topological $*$-algebras. Since it
is crucial to our work, we provide a proof for completeness.

\begin{lemma}
\label{LEM} Let $\mathcal{C}$ and $\mathcal{D}$ be two locally $C^{\ast }$%
-algebras with topologies defined by the family of $C^{\ast }$-seminorms $%
\{p_{\iota }\}_{\iota \in \Omega }$ and $\{q_{\iota }\}_{\iota \in \Omega }$%
, respectively. Let $\rho :\mathcal{C\rightarrow D}$ be a local contractive $%
\ast $-morphism with dense range such that $q_{\iota }\left( \rho (c)\right)
\leq p_{\iota }(c)$, for all $c\in \mathcal{C}$ and all $\iota \in \Omega $,
and $\mathcal{I}:=\ker (\rho )$. Then, there exists a unique local isometric 
$\ast $-isomorphism $\widetilde{\rho }:\overline{\mathcal{C}/\mathcal{I}}%
\rightarrow \mathcal{D}$ such that $\widetilde{\rho }\circ \sigma =\rho $,
where $\sigma :\mathcal{C}\rightarrow \overline{\mathcal{C}/\mathcal{I}}$ is
the canonical quotient map.
\end{lemma}

\begin{proof}
Since $\rho :\mathcal{C\rightarrow D}$ is a local contractive $\ast $%
-morphism with dense range such that $q_{\iota }\left( \rho (c)\right) \leq
p_{\iota }(c)$ for all $c\in \mathcal{C}$ and all $\iota \in \Omega ,$ for
each $\iota \in \Omega $, there exists a surjective $C^{\ast }$-morphism $%
\rho _{\iota }:\mathcal{C}_{\iota }\rightarrow \mathcal{D}_{\iota }$ such
that 
\begin{equation*}
\rho _{\iota }\left( \rho _{\iota }^{\mathcal{C}}(c)\right) =\rho _{\iota }^{%
\mathcal{D}}\left( \rho (c)\right) \ 
\end{equation*}%
for all $c\in C$. For each $\iota \in \Omega ,$ $\mathcal{I}_{\iota }=%
\overline{\rho _{\iota }^{\mathcal{C}}(\mathcal{I})}$ is a closed two-sided $%
\ast $-ideal of $\mathcal{C}_{\iota }$, and so, there exists a unique $%
C^{\ast }$-isomorphism $\widetilde{\rho }_{\iota }:\mathcal{C}_{\iota }/%
\mathcal{I}_{\iota }\rightarrow \mathcal{D}_{\iota }$ such that $\ 
\widetilde{\rho }_{\iota }\circ \sigma _{\iota }=\rho _{\iota }.$

\begin{equation*}
\begin{tikzcd} {C_{\iota}} && {D_{\iota}} \\ & {C_{\iota}/J_{\iota}}
\arrow["{\rho_{\iota}}", from=1-1, to=1-3] \arrow["{\sigma_{\iota}}"',
from=1-1, to=2-2] \arrow["{\widetilde{\rho}_{\iota}}", dashed, from=2-2,
to=1-3] \end{tikzcd}
\end{equation*}

Now, let $\iota _{1},\iota _{2}\in \Omega $ with $\iota _{1}\leq \iota _{1}.$
Since 
\begin{equation*}
\begin{aligned} \rho_{\iota_{1}\iota_{2}}^{D} \big(
\widetilde{\rho}_{\iota_{2}} ( \sigma_{\iota_{2}}(c_{\iota_{2}}) ) \big) &=
\rho_{\iota_{1}\iota_{2}}^{D} \big( \rho_{\iota_{2}} ( c_{\iota_{2}} ) \big)
\\ &= \rho_{\iota_{1}} \big( \rho_{\iota_{1}\iota_{2}}^{C} ( c_{\iota_{2}} )
\big) \\ &= \widetilde{\rho}_{\iota_{1}} \big( \sigma_{\iota_{1}} (
\rho_{\iota_{1}\iota_{2}}^{C} ( c_{\iota_{2}} ) ) \big) \\ &=
\widetilde{\rho}_{\iota_{1}} \big( \rho_{\iota_{1}\iota_{2}}^{C/J} (
\sigma_{\iota_{2}}( c_{\iota_{2}} ) ) \big), \quad (\forall)\ c_{\iota_{2}}
\in C_{\iota_{2}} \end{aligned}
\end{equation*}%
it follows that the family $\{\widetilde{\rho }_{\iota }\}_{\iota \in \Omega
}$ forms an inverse system of $C^{\ast }$-isomorphisms. Similarly, the
family $\{\widetilde{\rho }_{\iota }^{-1}\}_{\iota \in \Omega }$ also forms
an inverse system of $C^{\ast }$-isomorphisms. Let $\widetilde{\rho }
:=\varprojlim\limits_{\iota}\widetilde{\rho }_{\iota }$ and 
$\widetilde{\rho }^{-1}:=\varprojlim\limits_{\iota}
\widetilde{\rho }_{\iota }^{-1}.$ Then $\widetilde{\rho }$ and $\widetilde{%
\rho }^{-1}$are local isometric $\ast $-morphisms. Since $\widetilde{\rho }%
\circ \widetilde{\rho }^{-1}=$id$_{\mathcal{D}}$ and $\widetilde{\rho }%
^{-1}\circ \widetilde{\rho }=$id$_{\overline{\mathcal{C}/\mathcal{I}}}$, it
follows that $\widetilde{\rho }$ and $\widetilde{\rho }^{-1}$ are local
isometric $\ast $-isomorphisms.

\begin{equation*}
\begin{tikzcd} C && D \\ & {\overline{C/J}} \arrow["\rho", from=1-1, to=1-3]
\arrow["\sigma"', from=1-1, to=2-2] \arrow["{\widetilde{\rho}}"', dashed,
from=2-2, to=1-3] \end{tikzcd}
\end{equation*}
\end{proof}

\begin{theorem}
\label{R-Shilov boundary} Any unital local operator space $\mathcal{V}\subseteq
C^{\ast }(\mathcal{D}_{\mathcal{E}})$ has the local $\mathcal{R}$-Shilov
boundary.
\end{theorem}

\begin{proof}
Let $\left( C_{e\mathcal{R}}^{\ast }\left( \mathcal{V}\right) ,\kappa
\right) $ be the local $\mathcal{R}-C^{\ast }$-envelope of $\mathcal{V}$,
and let $\mathcal{C}\subseteq C^{\ast }(\mathcal{D}_{\mathcal{E}})$ be the
locally $C^{\ast }$-algebra generated by $\mathcal{RV}$. Since the pair $\left( 
\mathcal{C},i_{\mathcal{V}}\right) ,$ where $i_{\mathcal{V}}$ is the
inclusion $\mathcal{V\hookrightarrow C},$ is a local $\mathcal{R}$-$C^{\ast
} $-extension of $\mathcal{V}$, there exists a unique local contractive $%
\mathcal{R}-\ast $-morphism $\rho :\mathcal{C}\rightarrow C_{e\mathcal{R}%
}^{\ast }\left( \mathcal{V}\right) $ with dense range such that $\rho
\upharpoonright _{\mathcal{V}}=\kappa .$

Let $\mathcal{J}_{\mathcal{R}}:=\ker \rho .$ Clearly, $\mathcal{J}_{\mathcal{R}}$ is a
closed two-sided $\ast $-ideal of $\mathcal{C}$ and  an $\mathcal{R}$%
-module. We claim that $\mathcal{J}_{\mathcal{R}}$ is the local $\mathcal{R}$-Shilov
boundary for $\mathcal{V}$.

Since $\rho $ is an $\mathcal{R}$-map, we have $\Vert \rho (c)\Vert _{\iota
}\leq \Vert c\Vert _{\iota }$ for all $\iota \in \Omega $ and for all $c\in
C $. By Lemma \ref{LEM}, there exists a unital local isometric $\ast $%
-isomorphisms $\widetilde{\rho }$ such that the following diagram commutes%
\begin{equation*}
\begin{tikzcd} \mathcal{V} && C^*_{eR} \\ {C=C^*(\mathcal{V})} && {\overline{\mathcal{C}/\mathcal{J}_{\mathcal{R}}}}
\arrow["\kappa", from=1-1, to=1-3] \arrow[hook, from=1-1, to=2-1]
\arrow["\rho"', dashed, from=2-1, to=1-3] \arrow["\sigma"', from=2-1, to=2-3]
\arrow["{\widetilde{\rho}}"', dashed, from=2-3, to=1-3] \end{tikzcd}.
\end{equation*}

Since $\widetilde{\rho }\circ \sigma \upharpoonright _{\mathcal{V}}=\rho
\upharpoonright _{\mathcal{V}}=\kappa $, it follows that $\sigma
\upharpoonright _{\mathcal{V}}=\widetilde{\rho }^{-1}\circ \kappa $ is a
unital local complete isometry, and so $\mathcal{J}_{\mathcal{R}}$ is a
local boundary $\mathcal{R}$-ideal for $\mathcal{V}$. Moreover, $\left( 
\overline{\mathcal{C}/\mathcal{J}_{\mathcal{R}}},\sigma \upharpoonright _{%
\mathcal{V}}\right) $ is a local $\mathcal{R}-C^{\ast }$-envelope of $%
\mathcal{V}$, since $\widetilde{\rho }:\overline{\mathcal{C}/\mathcal{J}_{%
\mathcal{R}}}\rightarrow C_{e\mathcal{R}}^{\ast }\left( \mathcal{V}\right) $
is a unital local isometric $\mathcal{R}$-$\ast $-isomorphism such that $%
\widetilde{\rho }\circ \sigma \upharpoonright _{\mathcal{V}}=\kappa .$

Let $\mathcal{I}\subseteq \mathcal{C}$ be any local boundary $\mathcal{R}$%
-ideal for $\mathcal{V}$, and let $\upsilon :\mathcal{C}\rightarrow \overline{%
\mathcal{C}/\mathcal{I}}$ be the canonical $\ast $-morphism. Since $\upsilon
\upharpoonright _{\mathcal{V}}$ is a local complete isometry, $\left( 
\overline{\mathcal{C}/\mathcal{I}},\upsilon \upharpoonright _{\mathcal{V}%
}\right) $ is a local $\mathcal{R}-C^{\ast }$-extension of $\mathcal{V}$.
Now, since $\left( \overline{\mathcal{C}/\mathcal{J}_{\mathcal{R}}},\sigma
\upharpoonright _{\mathcal{V}}\right) $ is a local $\mathcal{R}-C^{\ast }$%
-envelope of $\mathcal{V}$, there exists a unique local contractive $%
\mathcal{R}-\ast $-morphism $\varphi :\overline{\mathcal{C}/\mathcal{I}}%
\rightarrow \overline{\mathcal{C}/\mathcal{J}_{\mathcal{R}}}$ with dense
range such that $\varphi \circ \upsilon \upharpoonright _{\mathcal{V}%
}=\sigma \upharpoonright _{\mathcal{V}}.$

If $x\in \mathcal{I}$, then 
\begin{equation*}
x+\mathcal{J}_{\mathcal{R}}=\varphi \left( x+\mathcal{I}\right) =\varphi
\left( 0+\mathcal{I}\right) =0+\mathcal{J}_{\mathcal{R}}
\end{equation*}%
and so, $x\in \mathcal{J}_{\mathcal{R}}$. Consequently, $\mathcal{I}\subseteq 
\mathcal{J}_{\mathcal{R}}$. Hence, $\mathcal{J}_{\mathcal{R}}$ contains
all local boundary $\mathcal{R}$-ideals, and so $\mathcal{J}_{\mathcal{R}}$
is the local $\mathcal{R}$-Shilov boundary for $\mathcal{V}$.
\end{proof}

\begin{corollary}
Let $\mathcal{V}\subseteq C^{\ast }(\mathcal{D}_{\mathcal{E}})$ be a unital
local operator space. Then $\mathcal{I}_{\mathcal{R}}\left( \mathcal{V}%
\right) =\mathcal{I}_{\mathcal{R}}\left( C^{\ast }(\mathcal{V})\right) $ if
and only if $\mathcal{V}$ has trivial local $\mathcal{R}$-Shilov boundary.
\end{corollary}

\begin{corollary}
\label{C. R-Shilov} Let $\mathcal{V}_{1}\subseteq C^{\ast }(\mathcal{D}_{%
\mathcal{E}})$ and $\mathcal{V}_{2}\subseteq C^{\ast }(\mathcal{D}_{\mathcal{%
E}})$ be two unital local operator spaces, and let $\mathcal{C}_{1}$ and $%
\mathcal{C}_{2}$ be the unital locally $C^{\ast }$-algebras generated by $%
\mathcal{RV}_{1}$ and $\mathcal{RV}_{2}$, respectively. If $\varphi :%
\mathcal{V}_{1}\rightarrow \mathcal{V}_{2}$ is a unital local complete order 
$\mathcal{R}$-isomorphism and both $\mathcal{V}_{1}$ and $\mathcal{V}_{2}$
have trivial local $\mathcal{R}$-Shilov boundary, then $\varphi $ extends
uniquely to a unital local isometric $\mathcal{R}-\ast $-isomorphism $\rho :%
\mathcal{C}_{1}\rightarrow \mathcal{C}_{2}.$
\end{corollary}

\begin{proof}
Since $\mathcal{V}_{1}$ and $\mathcal{V}_{2}$ have trivial local $\mathcal{R}
$-Shilov boundaries, the pairs $\left( \mathcal{C}_{1},i_{\mathcal{V}_{1}}\right) $ and $%
\left( \mathcal{C}_{2},i_{V_{2}}\right) $ are local $\mathcal{R}-C^{\ast }$%
-envelopes of $\mathcal{V}_{1}$ and $\mathcal{V}_{2}$, respectively. By 
Proposition \ref{uniqueness envelopes}, there exists a unique unital local
isometric $\mathcal{R}-\ast $-isomorphism $\rho :\mathcal{C}_{1}\rightarrow 
\mathcal{C}_{2}$ such that $\rho \upharpoonright _{\mathcal{V}_{1}}=\varphi
. $%
\begin{equation*}
\begin{tikzcd} {\mathcal{V}_{1}} && {\mathcal{V}_{2}} \\ {\mathcal{C}_{1}=C^*(\mathcal{V}_{1})} &&
{\mathcal{C}_{2}=C^*(\mathcal{V}_{2})} \arrow["\varphi", from=1-1, to=1-3] \arrow[hook,
from=1-1, to=2-1] \arrow[hook, from=1-3, to=2-3] \arrow["\rho", from=2-1,
to=2-3] \end{tikzcd}
\end{equation*}
\end{proof}

\begin{example}
Let $\{\mathcal{H};\mathcal{E}=\{\mathcal{H}_{n}\}_{n\geq 1};\mathcal{D}_{%
\mathcal{E}}\}$ be a Fr\'{e}chet quantized domain in the Hilbert space $%
\mathcal{H}$, and let $\mathcal{V}\subseteq C^{\ast }(\mathcal{D}_{\mathcal{E}})$
be a unital local operator space. If $\mathcal{V}$ has trivial local $%
\mathcal{R}$-Shilov boundary and $K^{\ast }(\mathcal{D}_{\mathcal{E}%
})\subseteq C^{\ast }\left( \mathcal{V}\right) $, then $\mathcal{I}_{%
\mathcal{R}}(\mathcal{V})=$ $C^{\ast }(\mathcal{D}_{\mathcal{E}})$.
\end{example}





\subsection*{Funding} This research did not receive any specific grant from funding agencies in the public, commercial, or not-for-profit sectors.

\subsection*{Data Availability} This paper has no associated data.
\subsection*{Declarations} The authors have no relevant financial or non-financial interests to disclose.

\end{document}